\documentclass[11pt]{article}
\usepackage{amsmath,latexsym,amssymb,amsfonts,amsbsy, amsthm}
\usepackage [latin1]{inputenc}
\usepackage{color, xcolor}

\usepackage{graphicx}
\usepackage{soul, color}
\usepackage[margin = 1.5in]{geometry}

\newtheorem{theorem}{Theorem}[section]

\newtheorem{definition}[theorem]{Definition}

\newtheorem{lemma}[theorem]{Lemma}

\newtheorem{remark}[theorem]{Remark}

\newcommand{\ud}{\,\mathrm{d}}
\newcommand{\p}{\ensuremath{\partial}}
\newcommand{\n}{\ensuremath{\nonumber}}
\newcommand{\eps}{\ensuremath{\varepsilon}}

\newcommand{\bigO}{\mathcal{O}}

\newcommand{\mcf}{\mathcal{F}}

\def\mau{\mathbf{u}}

\def\dv{\mathrm{div}}
\let\f=\frac

\let\va=\varepsilon

\let\f=\frac

\def\dv{\mathrm{div}}

\def\curl{\mathop{\rm curl}\nolimits}

\begin{document}
\title{\textbf{Stability and related zero viscosity limit
of steady plane Poiseuille-Couette flows with no-slip boundary condition}}

\author{ {\bf Song $\text{Jiang}^\ddag$ \qquad\quad Chunhui $\text{Zhou}^\dag$\footnote{Corresponding author}}
\\[4mm]
\small $^\dag$ Department of Mathematics, Southeast University,
Nanjing, 210096, China.\\
\small Email: zhouchunhui@seu.edu.cn\\
\small $^\ddag$ Institute of Applied Physics and Computational Mathematics,\qquad\quad\quad\\
\small  Beijing, China. Email: jiang@iapcm.ac.cn}
\date{}
\maketitle

\begin{abstract}
We prove the existence and stability of smooth solutions to the steady Navier-Stokes equations near plane Poiseuille-Couette
flow. Consequently, we also provide the zero viscosity limit of the 2D steady Navier-Stokes equations to the steady Euler equations.
First, in the absence of any external force, we prove that there exist smooth solutions to the steady Navier-Stokes equations
which are stable under infinitesimal perturbations of plane Poiseuille-Couette flow. In particular, if the basic flow is the Couette flow,
then we can prove that the flow is stable for any finite perturbation $o(1)$. Moreover, we also show that for any smooth shear flow
satisfying (\ref{0.3}), if we put a proper external force to control the flow, then we can also obtain a smooth solution of
the steady Navier-Stokes equations which is stable for infinitesimal perturbation of the external force. Finally,
based on the same linear estimates, we establish the zero viscosity limit of all the solutions obtained above
to the solutions of the Euler equations.
\end{abstract}

\bigskip
\renewcommand{\theequation}{\thesection.\arabic{equation}}
\setcounter{equation}{0}

\section{Introduction}

Hydrodynamic stability has been recognized as one of the central problems in fluid mechanics. There have been plenty of works
including theoretical, experimental and numerical results on the stability characteristics of different classes of basic flows.
These include flows in channels, boundary layers, jets and shear layers.
Concerning the stability of viscous unsteady shear flows, the authors in  \cite{BGM,BGM1,CLWZ,CWZ,MZ,Romanov,WZ} show
that the  2D or 3D Couette flow and pipe Poiseuille flow are linear stable for infinitesimal perturbations at high Reynolds number.
In \cite{GGN1}, the authors prove that generic plane shear profiles other than the plane Couette
flow are linearly unstable for sufficiently large Reynolds number. More recently, the authors in \cite{LMZ} use a new energy method to prove the instability of plane Couette flow for some perturbation of size $\va^{\f12-\delta}$ with any small $\delta>0$.

In this paper, we shall study the stability of shear flows in the {\it steady setting}. First we prove that, in contrast with
the unsteady case, a steady plane Poiseuille flow is stable under infinitesimal perturbations in the absence of external forces.
In fact, we can also prove that the basic flows of plane Poiseuille-Couette family defined in (\ref{0.2}) are stable under infinitesimal perturbations. In particular, if the basic flow is Couette flow, then through a formal asymptotic expansion including the Euler correctors and weak boundary layer correctors, we can show that the flow is stable for any finite perturbation $o(1)$.
 Moreover, we also show that any shear flow $\mau_0=(\mu(y),0)$ satisfying (\ref{0.3}), which is a solution to the steady
 Navier-Stokes equations with an external force $\mathbf{f}_0=(-\mu'',0)$, is stable under an infinitesimal perturbation
 of the external force.
 Finally, based on the same linear estimates, we prove the zero viscosity limit of all the obtained solutions to the solutions of
  the Euler equations.

  Let $\mathbf{U}_0=(U(y),0)$ be the basic flow of Poiseuille-Couette family  with
  \begin{equation}U(y)=\alpha_1y+\alpha_2y(2-y),\  \ 0\leq y\leq2, \label{0.2}
  \end{equation}
where $ \alpha_1,\alpha_2\geq0$ are given constants and assumed always to satisfy
$\alpha_1+\alpha_2>0$. If $\alpha_1\equiv0$, then we get the plane Poiseuille flow,
while $\alpha_2\equiv0$, the flow is called the plane Couette flow.

For the characteristic quantities of the flow: the density $\rho^*$, the viscosity $\mu$, the characteristic velocity such as the free
stream velocity $V$, and the characteristic length $d$, we consider the dimensionless spatial coordinates $x=x^*/d,\ y=y^*/d$,
$(x,y)\in\Omega=(0,L)\times (0,2)$, the dimensionless velocity $\mau^\va=\mau^*/V$
 and the dimensionless normal stress  $P^\va=P^*/\rho^*V^2$. The steady incompressible Navier-Stokes equations can be written
 into the following dimensionless form:
\begin{equation} \label{main.NS}
\begin{cases}
\bold{u}^\va \cdot \nabla \bold{u}^\va + \nabla P^\va - \va \Delta \bold{u}^\va =\bold{f}^\va,\\
\nabla \cdot \bold{u}^\va = 0,
\end{cases}
\end{equation}
where $\mathbf{u}^\va=(u^\va,v^\va)$, $\va=1/Re$, $Re =\frac{\rho^* V^* d}{\mu}$ is the Reynolds number, $P^\va$ is the pressure,
and $\bold{f}^\va=(f_1^\va,f_2^\va)$ is the external force.

Besides, we consider an Euler flow:
\begin{align} \label{shear.euler}
\bold{u}^0 = (\mu(y), 0),
\end{align}
which,  for any sufficiently smooth function $\mu$, 
satisfies the stationary Euler equations:
\begin{equation} \label{main.euler}
\begin{cases}
\bold{u}^0 \cdot \nabla \bold{u}^0 + \nabla P^0 = 0 \\
\nabla \cdot \bold{u}^0 = 0 \\
\bold{u}^0 \cdot \bold{n}|_{y = 0, y = 2} = 0.
\end{cases}
\end{equation}
Here the pressure $P^0$ is a constant, $\bold{n}$ denotes the unit normal vector on $\{y = 0\}$ and $\{y = 2\}$, and the boundary condition $\bold{u}^0 \cdot \bold{n} = 0$ is known as the no-penetration condition.

We are interested in the asymptotic behavior of $\bold{u}^\eps$ as $\eps \rightarrow 0$ with no-slip boundary condition on $y=0$ and $y=2$.
In the presence of boundaries, if there is a mismatch between the no-slip boundary condition for the Navier-Stokes equations
and the no-penetration boundary condition satisfied by the Euler equations, then the limit as  $\eps \to 0$
would be rectified by the presence of the Prandtl boundary layer. There have been many works on the well-posedness
of the Prandtl boundary layer equations. Concerning the instability of the Prandtl boundary layer
in the non-steady setting, we can refer, for instance, to \cite{AWXY,GD,GMM,GM,G,GGN,MW,LY} and the references therein.
Considering the well-posedness of steady Prandtl boundary layers,
we refer to \cite{GM1,GI,GI1,GN,Iyer,IM,IM1} and the references therein.

  In this article, we shall study the situation when the Euler flow $\bold{u}^0 = (\mu(y), 0)$ itself is  a shear flow
  satisfying the no-slip boundary condition:
\begin{align}\label{0.6}
\mu(0) = 0,\ \mu(2)=2\alpha_1,
\end{align}
and is a small perturbation of the basic flow $\mathbf{U}_0$ (cf. (\ref{0.2})).

  In the 2D steady setting, under the assumption that $\mu''/\mu$ vanishes at high order at $y=0,2$, and
in the absence of external forces, the authors in \cite{IZ} investigate the inviscid limit of weak solutions to
the steady Navier-Stokes equations in a channel around the shear flows $(\mu(y),0)$ with homogeneous no-slip boundary condition
on the rigid wall and the stress-free condition on the outflow part of the boundary.
The weak solutions in \cite{IZ} can not be strong ones since one has no adequate regularity
near the corners of the domain due to the stress free boundary condition.
In \cite{GM1},  under the proper control
of the external forces, the authors establish the $H^1$ stability of shear flows of Prandtl type with periodic boundary condition in  $x$-direction.

 The aim of this paper is to study the existence and stability of smooth solutions to the steady Navier-Stokes equations around the plane
 Poiseuille-Couette family under the perturbations of both viscous stresses and external forces. Except for the no-slip boundary
 condition on the rigid walls $y=0$ and $y=2$, we assume further that the curl of the flow is well controlled
on the inflow and outflow parts of the boundary.

Before stating the main results, we introduce the notation used throughout this paper.
\\[1mm]
{\sc Notation:} \ Let $G$ be an open domain in $\mathbb{R}^N$.
We denote by $L^p(G)$ ($p\geq 1$) the Lebesgue spaces, by $W^{s,p}(G)$ ($p\geq 1$) the Sobolev spaces with $s$ being a real number,
by $H^k(G)$ ($k\in \mathbb{N}$) the Sobolev spaces $W^{k,p}(G)$ with $p=2$, and
by $C^k(G)$ (resp. $C^k(\overline{G})$) the space of $k$th-times continuously differentiable functions in $G$ (resp. $\overline{G}$).
 We use $|\cdot|$ to denote the Lebesgue norm at the boundary $\partial\Omega$ throughout this paper.
 $\|\cdot\|_{k,p}$ stands for the standard norm in $W^{k,p}(G)$ and $\|\cdot\|$ for the norm in $L^2(G)$.
 We also use $|\cdot|_{L^\infty}$ to denote $|\cdot |_{L^\infty(\Omega)}=\text{ess~sup}_{\Omega}|\cdot |$.
 The symbol $\lesssim$ means that the left side is less than the right side multiplied by some constant.
Let $a\in {\mathbb R}$ be a real number, then $a+$ means any real number a little bigger than $a$.

We also define a smooth cut-off function $\chi(t)\in C^\infty([0,\infty))$ satisfying $|\chi|\leq1,\ |\chi|_{C^4}\leq C$ and
\begin{equation}\label{cutoff}\chi(t)=\begin{cases}1,\ 0\leq t\leq 1/2,\\0,\ t\geq1,\end{cases}\end{equation}
here $C>0$ is a finite constant.

Next, we give the boundary condition considered in this paper.  Precisely, we impose the following boundary condition:
\begin{eqnarray}\label{boundary1}
&&\ v_x^\va|_{x=0}=a_1;\ u^\va|_{x=0}=\mu(y)+a_2;\ v^\va|_{x=L}=a_3;\ \partial_x\curl\mau^\va|_{x=L}=a_4;\nonumber\\
&&u^\va|_{y=0}=v^\va|_{y=0}=v^\va_y|_{y=0}=v^\va|_{y=2}=v^\va_y|_{y=2}=0,\ u^\va|_{y=2}=2\alpha_1,
\end{eqnarray}
where $a_i$ ($i=1,...,4$) are pre-scribed functions which are sufficiently small.
For simplicity, in the following we shall assume $a_i=0$, $i=1,...,4$.

To study the asymptotics as $\va\to 0$, we shall use the following ansatz of asymptotic expansion:
\begin{eqnarray}\label{0.exp}\begin{cases}
u^\va=u_s+\va^{M_0} u\\
v^\va=v_s+\va^{M_0} v\\
P^\va=P_s+\va^{M_0} P,\end{cases}
\end{eqnarray}
where  $\mau_s,P_s$ are given function to be determined later,  $\mau_s=(u_s,v_s)$ satisfying $\dv\,\mau_s=0$,
$M_0$ is a constant sufficiently large, such that we can control the nonlinear terms in the a priori estimates. In this paper we shall
take $M_0=\f{11}8+$. Putting (\ref{0.exp}) into (\ref{main.NS}), we obtain that  the remainder solutions $(u,v,P)$ satisfy the following system:
\begin{eqnarray}
\dv\mau=0,\label{0.4}\\
u_su_x+u_{sy}v+u_{sx}u+v_su_y-\va\Delta u+\partial_xP=N_1(u,v)+F_1 \\
u_sv_x+v_sv_y+v_{sx}u+vv_{sy}-\va\Delta v+\partial_yP=N_2(u,v)+F_2,\label{0.5}
\end{eqnarray}
where $\mau=(u,v),\ N_1=-\va^{M_0}( v u_y+ uu_x),\
N_2=-\va^{M_0} (vv_y+ uv_x)$, $\mcf=(\mcf_u,\mcf_v)$, $\mathbf{F}=(F_1,F_2)=\mcf+\va^{-M_0}\mathbf f^\va$ and
\begin{eqnarray}
&&\mathcal{F}_u=-\va^{-M_0}[u_s\partial_xu_s+v_s\partial_yu_s+\partial_xP_s-\eps\Delta u_s]\label{remain1}\\
&&\mathcal{F}_v=-\va^{-M_0}[u_s\partial_xv_s+v_s\partial_yv_s+\partial_yP_s-\eps\Delta v_s]\label{remain2}.
\end{eqnarray}

Finally, let us define the space $\mathcal{X}$:
\begin{eqnarray}
\mathcal{X}=&& \{\mau=(u,v)\in H^3(\Omega)\times  H^3(\Omega)\; |\;\dv\mau=0,\ \|\mau\|_{\mathcal{X}}<\infty,\ u|_{y=0}=v|_{y=0}=0,\nonumber\\
&&u|_{y=2}=v|_{y=2}=u|_{x=0}=v_{x}|_{x=0}=v|_{x=L}=v_{xx}|_{x=L}=0\},
\end{eqnarray}
where the norm $\|\cdot\|_{\mathcal{X}}$ is defined by
$$\|\mathbf{u}\|_{\mathcal{X}}:=\|\sqrt{u_s}\nabla v\|+\va^{\f12}\|\sqrt{u_s}\nabla^2q\|+\va^{\f32}\|\nabla^3\mathbf{u}\|+|u_sq_y(0,\cdot)|, $$
with $q = v/u_s$.

\subsection{Main theorems}

In this section we state the main results. First, we shall show that in the absence of any external force,
if $\mau_0=(\mu(y),0)$ is a small perturbation of the basic flow $\mathbf{U}_0$ within order $O(\va^{\f38+})$, and
\begin{equation}\mu(y)>0 \; \text{ for }\ 0<y<2,\quad  \mu'(0)>0, \label{euler.1}\end{equation}
then there exists a unique solution to the steady Navier-Stokes equations around $\mau_0$
which will converge to $\mathbf{U}_0$ as $\va\rightarrow0$. More precisely, if we denote
\begin{equation}
\mau^\va=\mau_0+\va^{1+\f38+\gamma}\mau,\label{0.1}
\end{equation}
where $\gamma>0$ is a sufficiently small constant, then our first main result reads as follows.
\begin{theorem}\label{main.1}
Assume that $\mathbf{f}^\va\equiv0$ in (\ref{main.NS}), $0<L\ll1$ is a given constant,
 $\mu\in C^4([0,2])$ satisfies (\ref{0.6}), (\ref{euler.1}), then there is a constant $\alpha_0>0$, such that if
 \begin{equation}|\mu (y)-U(y)|_{C^4([0,2])}\leq \alpha_0 \va^{\f38+\gamma},\label{0.13}\end{equation}
there exists a unique solution $(\mau^\va,P^\va)\in H^4(\Omega)\times H^3(\Omega)$ to the system (\ref{main.NS}) with the remainder solution
$\mau=(u,v)$ defined in (\ref{0.1}) satisfying
\begin{equation}
\|v\|_{L^2}+\va^{\f12}\|\nabla^2v\|_{L^2}+\va\|u_{yy}\|_{L^2}+\va^{\f32}\|\nabla^3\mau\|_{L^2}
+\va^{\f52}\|\nabla^4\mau\|_{L^2}\leq C\alpha_0,\label{th.0.1}
\end{equation}
and consequently
\begin{eqnarray}
&&\va^{-\f12}|u^\va-\mu|_{L^\infty(\Omega)}+\va^{-\gamma}\|u^\va-\mu\|_{H^2(\Omega)}\leq C\va^{\f38},\label{th.0.4}\\
&&\va^{-\f14}|v^\va|_{L^\infty(\Omega)}+\va^{-\gamma}\|v^\va\|_{H^2}\leq C\va^{\f78},\label{th.0.2}
\end{eqnarray}
where  the constant $C$ does not depend on $\va$.
\end{theorem}

Next, in addition to (\ref{euler.1}), if we assume that $\mu''$ degenerates  rapidly,  when $ y\rightarrow0$,
i.e., we have to take $\alpha_2\equiv0$ and thus the basic flow is  Couette flow, then the flow is stable within any finite disturbance $o(1)$ of the basic flow.

We denote
\begin{equation}
\mau^\va=\mau_s+\va^{1+\f38+\gamma}\mau,\label{0.9}
\end{equation}
where $\mau_s$ is a known function defined in (\ref{asy:1:1:1}). Our second result reads as
\begin{theorem}\label{main.2}
In addition to the assumptions in Theorem \ref{main.1}, we also assume that $\alpha_2=0$ in (\ref{0.2}), $k>0$
is a suitably large integer, and
\begin{equation}|\f{\mu''}{\mu}|_{L^\infty} \ \text{is\ sufficiently \ small,}\ \ \ \label{0.8}  \end{equation}
then there is a constant $\alpha_0>0$, such that if
\begin{equation}|\f{\mu'''}{\mu}|_{C^{k}([0,2])}\leq \alpha_0,\label{th.0.3}\end{equation}
  there exists a unique solution $(\mau^\va,P^\va)\in H^4(\Omega)\times H^3(\Omega)$ to the system (\ref{main.NS}) satisfying
  \begin{equation}
  \|\mau^\va\|_{H^2}+\|P^\va\|_{H^1}\leq C\alpha_0,
  \end{equation}
  and
\begin{equation}
\|v\|_{L^2}+\va^{\f12}\|\nabla^2v\|_{L^2}+\va\|u_{yy}\|_{L^2}+\va^{\f32}\|\nabla^3\mau\|_{L^2}+\va^{\f52}\|\nabla^4\mau\|_{L^2}\leq C\alpha_0,
\end{equation}
where $\mau=(u,v)$ is the remainder solution defined in (\ref{0.9}).
Moreover, it holds that
\begin{eqnarray}
 |u^\va-\mu|_{L^\infty}+|v^\va|_{L^\infty}\leq  C\va ,\label{th.0.5}
\end{eqnarray}
where $C$ is a constant independent of  $\va$.
\end{theorem}

Finally,  if we put a suitable external force $\mathbf{f}^\va$
to ''rewind'' the flow, then we can obtain a unique smooth solution to the system (\ref{main.NS}) around any basic shear flow
$\mathbf{u}^0(y)=(\mu(y),0)$, satisfying
\begin{equation}
\mu'(0)>0,\ \mu(y)>0,\ \text{for}\ 0<y<2,\ \ |\mu|_{C^4}\leq K_0 \ \ \text{in}\ \Omega,\label{0.3}
\end{equation}
where $K_0>0$ is a fixed constant.
More precisely, setting $$\mathbf{g}^\va=\mathbf{f}^\va+\va (\mathbf{u}^0)'',\ \mau^\va=\mathbf{u}^0+\va^{1+\f{3}{8}+\gamma}\mau,$$
where $\mau=(u,v),\ \mathbf{f}^\va=(f_1^\va,f_2^\va),\ \mathbf{g}^\va=(g_1^\va,g_2^\va)$. Our last theorem reads as
\begin{theorem}\label{main.3}
Let $0<L\ll1$ be a given constant,  $\mu(y)$ is a smooth function satisfying (\ref{0.3}). Then, there exists a constant $\alpha_0>0$, such that if
\begin{equation}\|\mathbf{g}^\va\|_{H^2}\leq \alpha_0 \va^{\f{11}8+\gamma},\label{0.12}\end{equation}
 there exists a unique solution $(\mau^\va,P^\va)\in H^4(\Omega)\times H^3(\Omega)$ to the system (\ref{main.NS}), satisfying (\ref{th.0.1})-(\ref{th.0.2}).
\end{theorem}
\begin{remark}
The zero viscosity limit of the solutions in Theorems 1--3 to the solutions of the steady Euler equations follow immediately from
(\ref{th.0.4}), (\ref{th.0.2}) and (\ref{th.0.5}). In particular, we have
$$ \mau^\va\rightarrow\mau_0\ \text{ in }\ H^2(\Omega) $$
in Theorems 1 and 3, while in Theorem 2, we only have
\begin{eqnarray}\mau^\va\rightarrow\mau_0\ \text{ in } H^s(\Omega), \ 0<s<2.\label{0.10}\end{eqnarray}
However, if we further assume
$$|\f{\mu'''}{\mu}|_{C^{k}([0,2])}\leq \alpha_0\va^{0+},$$
then we can take $s=2$ in (\ref{0.10}).
\end{remark}
\begin{remark}
In Theorem \ref{main.2},  $\mau_0=(\mu,0)$ could be any shear flow with $\mu$ be a smooth function satisfying (\ref{euler.1}),
(\ref{0.8}) and (\ref{th.0.3}). In particular, if $\mu(0)=\mu(2)=0$, we need one more condition that $\mu''$ degenerates rapidly near $y=2$.
In this case, we can use the same scale $Y=\va^{-\f13}\tilde{y}$ ($\tilde{y}=y$ near $y=0$, $\tilde{y}=2-y$ near $y=2$)
in the construction of the weak boundary layer correctors in Section 2.
\end{remark}
 \begin{remark}In Theorem \ref{main.2}, although the remainders $(\mcf_u,\mcf_v)$ defined in (\ref{remain1})-(\ref{remain2})
 can be arbitrary small,  we can only obtain the uniform-in-$\va$ boundedness
 of $\|\mau^\va\|_{H^2}$ as a result of the weak boundary layer correctors  in the approximate solution $\mau_s$.
 \end{remark}
Let us give a few comments on our results. Considering solutions to the steady Navier-Stokes equations in the absence of external forces,
the class of strictly parallel flows $\mathbf{U}_0=(\mu(y),0)$ is limited, since $\mu(y)$ has to satisfy the equation of motion: $$\va\f{d^2\mu}{dy^2}=\f{dP}{dx}.$$
 This includes two important special cases:
 $${\it Plane\ Couette\ flow}:\qquad \mu(y)=y, \ P=\text{constant},\ 0<y<2,$$
 and
$${\it Plane\ Poiseuille\ flow}:\qquad \mu(y)=y(2-y), \ \f{dP}{dx}=\text{constant},\ 0<y<2.$$
We have therefore obtained a  family of strictly parallel flows, say the plane Poiseuille-Couette flow $\mathbf{U}_0=(U(y),0)$ with
  $U(y)$ defined in (\ref{0.2}).

Compared with \cite{GM1,GI}, where the authors studied the stability of the Prandtl expansions for the steady Navier-Stokes equations,
 the Euler flow $\bold{u}^0 = (\mu(y), 0)$  in this paper satisfies the no-slip boundary condition:
$\mu(0) = 0,\ \mu(2)=2\alpha_1$. Thus, there would be no strong boundary layers around the rigid walls $y=0$ and $y=2$.
In the absence of any external force, we could construct an Euler corrector to balance the perturbation of the internal viscous
stress of the fluid.  Due to the no-slip boundary condition $\mu(0) = 0$, the equation for the Euler corrector would have
a degenerate coefficient. It seems that for  Poiseuille-Couette flows with $\alpha_2>0$, we can not continue the expansion
to obtain arbitrary small remainders due to the eigenvalue problem of the Laplace operator. However,
if  $\alpha_2=0$, and $\mu''$ degenerate rapidly near $y=0$, then we can continue the expansion
to obtain arbitrary small remainders by using weak boundary layer correctors to rectify the mismatch between the no-slip boundary
 condition for the Navier-Stokes equations and the no-penetration boundary condition for Euler correctors.
 So, we can prove that the Couette flow is stable under finite small perturbation.

 Our proof is inspired by the method in \cite{GI}. Compared with the boundary layer profiles in \cite{GI},
 we would have one more boundary $y=2$ with no-slip boundary condition, and as mentioned above,
 the equations for the Euler correctors would have degenerate coefficients
 near the rigid walls. If the basic flow is a Couette flow, considering the weak boundary layer correctors  in the expansion,
 the dominating equation is a  linear parabolic equation with a degenerate coefficient $\mu(y)\sim y$ near the boundary $y=0$ and
  non-degenerate coefficients near $y=2$. So in the process of multi-scale analysis, we have to use different scales to balance
 the weak boundary layer terms, i.e., we utilize the scale $Y_- := y/\eps^{\frac{1}{3}}$ near $y=0$ and
 $Y_+ :=(2-y)/\eps^{\frac{1}{2}}$ near $y=2$. Consequently, the boundary conditions for the corresponding Euler correctors
 $\mau_e^{i,\pm}$ are also different near the rigid walls. Moreover, the linear estimates in \cite{GI} failed in obtaining the third order derivative $v_{yyy}$ as we do not have a good sign for $u_{sy}$ on the upper bound $y=2$.
 Instead, we shall use the theory of the Dirichlet boundary value problem for biharmonic equations as well as the Sobolev imbedding theory
 to close the estimates. Finally, let us describe the main steps in the proof of the  theorems.

 In Section 2, we give the construction of the asymptotic solution $\mau_s$ when $\alpha_2=0$ in the absence of any external force. First, we  use an  Euler corrector $\mau_e^1$ to offset the perturbation of the viscous stress of the flow.  Then  we use  weak boundary layer correctors to rectify  the mismatch between the no-slip condition for Navier-Stokes equations  and the no-penetration condition for Euler correctors.
Thanks to the condition  that $\mu''$ degenerates rapidly when $ y\rightarrow0$,
we can continue to make as many expansions as we want and the remainders $(\mcf_u,\mcf_v)$ can be arbitrary small.
For the weak boundary layer correctors  in the expansion, as mentioned above,
we use the scale $Y_- := y/\eps^{\frac{1}{3}}$ near $y=0$ and the scale $Y_+ :=(2-y)/\eps^{\frac{1}{2}}$ near $y=2$. Consequently, we have to construct the  weak boundary layer correctors $u_p^{i,\pm}$ separately.

In Section 3, we study the linearized system. First by introducing the stream function, we prove the existence of solutions to
the linearized system. Then, we derive the uniform-in-$\va$ estimates of solutions to the linearized system.
 We first use $v$ as a multiplier to act on the curl equation  to obtain  the $L^2$ norm of $v$
  bounded in order $O(1)$. Then, we use the multiplier $q_x= (v/{u_s})_x$ to act on the curl equation to get the estimates
 for second order derivatives. Broadly speaking, the multiplier $q_x$ works well in controlling the convection terms,
  but also brings complicated calculations in controlling the viscous terms.
  Next, as mentioned above, we can not use the method in \cite{GI} to obtain the estimates for the third order derivatives
  of the solutions as we do not have  good signs for $u_{sy}$ on the upper bound $y=2$.
  Instead, we shall employ the multiplier $v_{xxx}$ to act on the curl equation to obtain a bound for $\|\nabla^2v_x\|$
  which implies the boundedness of the convection terms in the biharmonic equation for $v$. Finally,
  by the compatibility condition from the curl equation, we can obtain the estimates
   of fourth order derivatives of $v$ by employing the theory of the Dirichlet boundary value problem for biharmonic equations.
  In this way, we can close the estimates by employing a careful bootstrap argument and the Sobolev imbedding theory.

Finally in Section 4, based on the linear estimates in Section 3, we prove the main theorems by the contract mapping principle.
We point out here that the only difference between the proofs of Theorems 1-3 lies in the remainders $\mcf=(\mcf_u,\mcf_v)$ in the linearized
system. Besides, the restriction on the exponent $M_0$ in the expansion (\ref{0.exp}) mainly comes from controlling the nonlinear
terms in Subsection 3.2.

\renewcommand{\theequation}{\thesection.\arabic{equation}}
\setcounter{equation}{0}

\section{Formal asymptotic expansion around the Couette flow}

The goal of this section is to construct approximate solutions $(\mau_s,P_s)$ in (\ref{0.exp}). In the case when $\alpha_2>0$ in the absence of any external force, or there is a suitable external force $\mathbf{f}^\va$  to control the flow,  we will  directly take $\mau_s=\mau^0$. So in this section we will only study the special case when $\alpha_2=0$ in the absence of any external force, i.e., the basic flow is Couette flow. Without loss of generality, we will assume $\alpha_1=1$.  Due to the fact that the Euler profile $\mau^0$  satisfies the no-slip boundary condition, there would be no strong boundary layers near the boundaries $y=0$ and $y=2$.

 First, we shall use an  Euler corrector $\mau_e^1$ to offset the perturbation of the viscous stress of the fluid.
Then  there would be, similar to the case of Prandtl boundary layers,  a mismatch between the no-slip condition for Navier-Stokes equations  and the no-penetration condition  for the Euler corrector. This mismatch can be rectified by the presence of the weak  boundary layer corrector. As the dominating equation for the weak boundary layer corrector  is a linear parabolic equation including a degenerate term near $y=0$, in the process of multi-scale analysis near the boundaries, we will broke the boundary layer profile into two components, one supported near $y = 0$ with  scale $Y_- := y/\eps^{\frac 1 3}$,  $\mau_p^{1,-}$, and one supported near $y = 2$ with scale $Y_+ := (2-y)/\eps^{\frac 1 2}$,  $\mau_p^{1,+}$. Thanks to the degeneration of $\mu''$ near $y=0$, we can continue the  construction of $\mau_e^{i,\pm}$ and $\mau_p^{i,\pm,}$  till the remainders defined in (\ref{remain1})-(\ref{remain2}) are small enough.

 In what follows,
 the Eulerian profiles are functions of $(x,y)$, whereas the boundary layer profiles are functions of $(x,Y)$, where
\begin{align}
Y =
\left\{
\begin{aligned} \label{Y:defn}
&Y_+ := \frac{2-y}{\eps^{\frac 1 2}} \text{ if } 1 \le y \le 2, \\
&Y_- := \frac{y}{\eps^{\frac 1 3}} \text{ if } 0 \le y \le 1.
\end{aligned}
\right.
\end{align}

 As we  have used different scales $Y_{\pm}$ in the construction of the weak boundary layer profiles, we will have to use different Euler correctors $\mau_e^{i,\pm}$ in the following expansion. The expansion will be continued till the remainders are small enough. More precisely,
 we  expand the solution in $\va$ as:
\begin{equation}\label{asy:1:1:1} \begin{cases}
u^\eps=\mu+\va u_e^1+\va u_p^{1,+}+\va u_p^{1,-}+ \sum_{i = 2}^{M} \eps^{1 + \frac{i-1}{2} } (u^{i,+}_e + u^{i,+}_p)\\
\qquad+ \sum_{i = 2}^{M} \eps^{1 + \frac{i-1}{3} } (u^{i,-}_e + u^{i,-}_p) +\eps^{M_0}u\\
\quad \triangleq u_s + \eps^{M_0}u\\
v^\eps=\va v_e^1+\va^{\f32} v_p^{1,+}+\va^{\f43} v_p^{1,-}+ \sum_{i = 2}^{M} \eps^{1 + \frac{i-1}{2} } (v^{i,+}_e + \va^{\f12}v^{i,+}_p)\\
\qquad+ \sum_{i = 2}^{M} \eps^{1 + \frac{i-1}{3} } (v^{i,-}_e + \va^{\f13}v^{i,-}_p)+\eps^{M_0}v \\
\quad \triangleq v_s + \eps^{M_0}v,\\
P^\eps =C+\va P_e^1+\sum_{i = 1}^5 \va^{1+\f{1+i}3}P^{i,a,-}_P+\sum_{i = 2}^{M} \eps^{1 + \frac{i-1}{3} } P^{i,-}_e +\sum_{i = 1}^3 \va^{1+\f{1+i}2}P^{i,a,+}_P\\
\qquad+\sum_{i = 2}^{M} \eps^{1 + \frac{i-1}{2} } P^{i,+}_e +\eps^{M_0}P \\
\quad \triangleq P_s + \eps^{M_0} P,\\
\end{cases}
\end{equation}
where $C$ is any constant, $M$ is an integer large enough, for example, $M=10$ would be enough.

We shall also introduce the notation
\begin{equation}\label{euler.0}\begin{cases}
u_s^E := \mu +\va u_e^1 +\sum_{i = 2}^{M} \eps^{1 + \frac{i-1}{3}} u^{i,-}_e +\sum_{i = 2}^{M} \eps^{1 + \frac{i-1}{2} } u^{i,+}_e, \\ v_s^E :=\va v_e^1+ \sum_{i = 2}^{M} \eps^{1 + \frac{i-1}{2} } v^{i,+}_e+\sum_{2 = 1}^{M} \eps^{1 + \frac{i -1}{3}} v^i_e.\end{cases}
\end{equation}

The main result of this section reads as:
\begin{theorem}\label{thm2.2}
Assume that $\mu$  is a smooth function satisfying conditions in Theorem \ref{main.2}, then we can construct a triple  $(u_s,v_s, P_s)$ with formula defined in (\ref{asy:1:1:1}) satisfying:
\begin{equation}\dv\mau_s=0\ \text{in}\ \Omega,\ u_s=v_s=0,\ \text{on}\ y=0;\ \ u_s=2\alpha_1,\ v_s=0,\ \text{on}\ y=2,\label{2.2.1}\end{equation}
and
\begin{eqnarray}
&&|\p_x u_s|+ |\p_y v_s| + |u_s - \mu|  \lesssim \min \{ O(\va^{\f23})y,O(\va^{\f12})(2-y), O(\eps) \},  \label{2.2.3}\\
&&\va^{-\f12}|\p_y u_s - \mu'| +|u_{syy}|+|u_{syyy}\chi(y)|+\va^{\f12}|u_{syyy}\chi(2-y)|\lesssim \alpha_0,  \label{2.2.4}\\
&&|\p_x^l v_s| \lesssim \min\{\eps y,\va(2-y)\} \text{ for } l \ge 0,\label{2.2.5}
\end{eqnarray}
where $\chi(y)$ is the cut-off function defined in (\ref{cutoff}). Moreover, the remainders defined in (\ref{remain1}) and (\ref{remain2}) satisfying
\begin{eqnarray}
&&\va^{M_0}\mathcal{F}_u = \mathcal{T}_{u,\eps^2} + \mathcal{F}_{u,\eps^{3}}\label{2.2.6}\\
&&\va^{M_0}\mathcal{F}_v = \mathcal{T}_{v,\eps^2} + \mathcal{F}_{v, \eps^{3}},\label{2.2.7}
\end{eqnarray}
with estimate:
\begin{align} \label{2.2.8}
\| \mathcal{T}_{u,\eps^2} \|_{H^2}+\| \mathcal{T}_{v,\eps^2}\|_{H^2}+\| \mathcal{F}_{u, \eps^3} \|_{H^2} + \| \mathcal{F}_{v, \eps^3} \|_{H^2} \lesssim \eps^2\alpha_0.
\end{align}

\end{theorem}

\begin{remark}
In the case when $\alpha_2>0$ in the absence of any external force, as the Euler profile $\mau^0$ satisfies
the no-slip boundary condition, we have $v=0$ on $y=0$ and $y=2$. For this case we can still construct the first Euler
corrector $\mau_e^1$ and the corresponding boundary layer correctors similar to the case $\alpha_2=0$. However,
as the equations for the Euler corrector have degenerate coefficients near $y=0$, we can not continue the expansion
due to the eigenvalue problem of the Laplace operator. In fact, if $v_e\neq 0$ on $y=0$, since $\mu''\sim-2\alpha_2$,
we shall have $\mu''/\mu\in (-\infty,0)$ in $\Omega$. Thus, the equation (\ref{divide.1}) is not well-posed.
 So in this case, the first weak boundary layer corrector is also the finally one, which do not degenerate when $Y$ tends to infinity.
 As a result, we can not obtain smaller remainders. So, we have to take $\mau_s=\mau_0$ directly when $\alpha_2>0$.
\end{remark}

\subsection{Euler correctors}

In this subsection we will study the construction of Euler correctors. The equations satisfied by the first Euler corrector are obtained by collecting the $\bigO(\eps)$ order Euler terms from (\ref{remain1}) - (\ref{remain2}), and is now shown:
\begin{equation} \label{euler.4}
\begin{cases}
  \mu \p_x u^1_e + \mu' v^1_e + \p_x  P^1_e = \mu''(y) \\
  \mu \p_x v^1_e + \p_y P^1_e = 0, \\
  \p_x u^1_e + \p_y v^1_e = 0, \\
  \partial_xv^1_e|_{x = 0} = 0,  v^1_e|_{x = L} = 0, \quad  v^1_e|_{y = 0} = v^1_e|_{y = 2} = 0.
\end{cases}
\end{equation}

By going to the vorticity formulation, we arrive at the following problem:
\begin{align}
-\mu \Delta v^1_e + \mu'' v^1_e = \mu'''(y), \hspace{3 mm} \hspace{3 mm} u^1_e := -\int_0^x v^1_{ey}(s,y)ds.\label{cou.1}
\end{align}
Then we divide (\ref{cou.1}) by $\mu$ to obtain:
\begin{align} \label{divide.1}
-\Delta v^1_e + \frac{\mu''}{\mu} v^1_e = \frac{\mu'''}{\mu}.
\end{align}

Thanks to the degeneration of $\mu''$ near $y=0$, If we take $|\f{\mu''}{\mu}|_{L^\infty}$ small enough, equation (\ref{divide.1}) will be well-posed. Next we turn to the construction of the rest Euler correctors. Due to the different scales of the weak boundary layer correctors near $y=0$ and $y=2$, we will have to use different Euler correctors to rectify the velocity $v_p^{i,\pm}$ from weak boundary layers correctors.  The system satisfied by the rest Euler correctors are shown as $(i=2,...M)$:
\begin{equation} \label{v2e.BVP}
\begin{cases}
  \mu \p_x u^{i,+}_e + \mu' v^{i,+}_e + \p_x  P^{i,+}_e = 0  \\
  \mu \p_x v^{i,+}_e + \p_y P^{i,+}_e = 0, \\
  \p_x u^{i,+}_e + \p_y v^{i,+}_e = 0, \\
 \partial_x v^{i,+}_e|_{x = 0} =0,\  v^{i,+}_e|_{x = L}=0,\ v^{i,+}_e|_{y = 2} = -v^{i-1,+}_p|_{y=2}, \hspace{3 mm} \partial_yv^{i,+}_e|_{y = 0} =0.
\end{cases}
\end{equation}
and
\begin{equation} \label{v2e.BVP1}
\begin{cases}
  \mu \p_x u^{i,-}_e + \mu' v^{i,-}_e + \p_x  P^{i,-}_e = 0  \\
  \mu \p_x v^{i,-}_e + \p_y P^{i,-}_e = 0, \\
  \p_x u^{i,-}_e + \p_y v^{i,-}_e = 0, \\
 \partial_x v^{i,-}_e|_{x = 0} =0,\  v^{i,-}_e|_{x = L}=0,\ \partial_yv^{i,-}_e|_{y = 2} = 0, \hspace{3 mm} v^{i,-}_e|_{y = 0} =-v^{i-1,-}_p|_{y=0}.
\end{cases}
\end{equation}
Going to vorticity produces the homogeneous system:
\begin{align}
- \Delta v^{i,\pm}_e + \f{\mu''}{\mu} v^{i,\pm}_e =0,\hspace{3 mm} u^{i,\pm}_e := -\int_0^x v^{i,\pm}_{ey}(s,y)ds.\label{asy.7}
\end{align}

This procedure contributes the error terms to the remainders
\begin{align}\label{euler.2}
\mathcal{C}_{Euler,u} := & (u_s^E - \mu) \p_x (u_s^E - \mu) + v_s^E \p_y (u_s^E - \mu) - \eps \Delta (u_s^E - \mu) , \\\label{euler.3}
\mathcal{C}_{Euler,v} := & (u_s^E - \mu) \p_x v_s^E + v_s^E \p_y v_s^E - \eps \Delta v_s^E,
\end{align}
where $u_s^E,v_s^E$ are defined in (\ref{euler.0}). We can observe that $\mathcal{C}_{Euler,u}, \mathcal{C}_{Euler,v}$ are in order $O(\eps^2)$.

We point out  that as we put the boundary conditions $\partial_yv^{i,-}_e = 0$ on $y=2$, and $\partial_yv^{i,+}_e = 0$ on $y=0$ for $i=2,3,...,M$, by the divergence free condition and the maximum principle of elliptic equations we can prove that $u_e^{i,+}=v_e^{i,+}=0$ on $y=0$, i.e. the Euler correctors $\mau_e^{i,+}$ will not cause the mismatch of the velocity on $y=0$. On the other hand, we have  $u_e^{i,-}=v_e^{i,-}=0$ on $y=2$, i.e. the Euler correctors $\mau_e^{i,-}$ will not cause the mismatch of the velocity on $y=2$.

 Precisely, we have the following Lemma for Euler correctors:
\begin{lemma}\label{lem2.2}
Assume  that $\mu(y)$ satisfies the conditions in Theorem \ref{thm2.2},  then there exists a unique triple $(u_e^1,v_e^1,P_e^1)$ satisfying system (\ref{euler.4}), and unique triples
$(u^{i,\pm}_e,v^{i,\pm}_e,P^{i,\pm}_e)$, $i=2,...,M$, satisfying system   (\ref{v2e.BVP}) and (\ref{v2e.BVP1}). Besides, we have
\begin{equation}
u_e^{i,+}=v_e^{i,+}=0\ \text{on} \ y=0;\ \ u_e^{i,-}=v_e^{i,-}=0\ \text{on} \ y=2,\ i=2,...M,\label{asy.6}
\end{equation}
and
\begin{align}\label{asy.5}
\| \p_x^l \p_y^m \{ u^{1}_e, v^{1}_e ,P^{1}_e\}\|+\| \p_x^l \p_y^m \{ u^{i,\pm}_e, v^{i,\pm}_e ,P^{i,\pm}_e\}\|
\leq C(m,l)|\frac{\mu'''}{\mu}|_{C^k},
\end{align}
where $m,l\geq0$ are integers satisfying $m+l\leq k+2$.
\end{lemma}
\begin{proof}
First of all, under the assumption that $|\f{\mu''}{\mu}|_{L^{\infty}}$ is small enough, the existence of $(u_e^1,v_e^1,P_e^1)$, $(u^{i,\pm}_e,v^{i,\pm}_e,P^{i,\pm}_e)$ and estimates (\ref{asy.5}) can be easily obtained by theory of elliptic equations.

Then we need to prove $u_e^{i,+}=v_e^{i,+}=0$ on $y=0$ and $u_e^{i,-}=v_e^{i,-}=0$ on $y=2$.

We will first prove $u_e^{i,+}=v_e^{i,+}=0$ on $y=0$ by weak maximum principle of elliptic equations.

 First, by (\ref{asy.7}), we easily have
  $$u_e^{i,+}(x,0)=-\int_0^xv_{ey}^{i,+}dx=0.$$
Next, for fixed $\sigma>0$ small enough, since we have already known by (\ref{asy.5}) that $v_e^{i,+}$  is bounded, we can define a barrier function
$$\phi^\sigma=K(y+\sigma)^\alpha,\ 0<\alpha<1,$$
where $K$ is a constant satisfying
$$K\geq C(|\f{\mu''}{\mu}|_{L^\infty}+\alpha_0)$$
with $C$ being a sufficiently large constant.
Then we have
\begin{equation}\Delta \phi^\sigma=\alpha(\alpha-1)K(y+\sigma)^{\alpha-2}<0.\label{asy.1}\end{equation}
Now we cut $v_e^{i,+}$ by the cut-off function $\chi(y)$ defined in (\ref{cutoff}).
Thanks to the boundary conditions $\partial_xv_e^{i,+}|_{x=0}=0$, we can make an even extension of $\chi(y)v_e^{i,+}$ with respect to $x=0$ and the new function, still denoted by  $\chi(y)v_e^{i,+}$,  belong to $C^2(\Omega^*)\cap C^0(\bar \Omega^*)$, here $\Omega^*=(-L,L)\times(0,2)$.
Recalling that $K$ is large enough,  we have by (\ref{asy.7}) and (\ref{asy.1})
$$\Delta(\chi v_e^{i,+}-\phi^\sigma)\geq 0\ \text{ in}\ \Omega^*.$$
Besides, direct computation shows that $$ \partial_y(\chi v_e^{i,+}-\phi^\sigma)|_{y=0}<0.$$
So we can not obtain the maximum value of $\chi v_e^{i,+}-\phi^\sigma$ on $y=0$. More over, we also have
$$\ (\chi v_e^{i,+}-\phi^\sigma)|_{\{y=2\}\cup\{x=L\}\cup\{x=-L\}}<0.$$
 The weak maximum principle for elliptic equations implies that
\begin{equation*}\chi v_e^{i,+}-\phi^\sigma\leq0\ \text{ in}\ \Omega^*.\end{equation*}
If we replace $\chi v_e^{i,+}$ by $-\phi^\sigma$, and $\phi^\sigma$ by $-\chi v_e^{i,+}$ above, we can obtain
$$-\phi^\sigma\leq-\chi v_e^{i,+}\ \text{ in}\ \Omega^*.$$
Combing above, we have proved that for any $0<\alpha<1$
$$|v_e^{i,+}|\leq K(y+\sigma)^\alpha,\ \text{when} \ 0\leq y\leq \f12,$$
for any $\sigma>0$ small enough, which implies immediately $v_e^{i,+}=0$ on $y=0$. Similarly we can also prove $u_e^{i,-}=v_e^{i,-}=0$ on $ y=2$ and we complete the proof of the Lemma.
\end{proof}

\subsection{Weak boundary layers correctors}

 In this subsection, we will only give the construction for $\mau^{i,+}_p$ near $y=2$. The construction of  $\mau_p^{i,-}$ near the boundary $y=0$ is exactly the same as the process in \cite{IZ1}. For the completeness of the article, we would give the sketch of the construction of  $\mau_p^{i,-}$ in Appendix.

\subsubsection{ Construction of $\mau^{i,+}_p$}

As there would be no degeneration in the coefficients of the linear parabolic equations for the weak boundary layer correctors near $y=2$, we would use the scale $Y_+=\f{2-y}{\va^{\f12}}$ here.
Using $(\ref{euler.4})$, the leading $\bigO(\eps)$ order boundary layer terms from (\ref{remain1}) are:
$$\mu \p_x u^{1, +}_p+\partial_xP^{1,+}_p - \p_{Y_+Y_+}u^{1, +}_p = 0.$$
The leading $\bigO(\eps^{\f12})$ order boundary layer terms from (\ref{remain2}) is:
$$\p_{Y_+} P^{1,+}_p = 0.$$
It is easy to obtain that $P^{1,+}_p\equiv0$, and we consider the following system(as $\mu(2)=2$)
\begin{eqnarray} \label{pr.BVP.1}
\begin{cases}
2 \p_x u^{1,0, +}_p - \p_{Y_+Y_+}u^{1,0, +}_p = 0, \\
u^{1,0, +}_p|_{x = 0} =0 , \hspace{3 mm} u^{1,0, +}_p|_{Y_+ = 0} = -u^1_e|_{y = 2}, u^{1,0, +}_p|_{Y_+ \rightarrow \infty} = 0 \\
v^{1,0, +}_p = \int_{Y_+}^\infty \p_x u^{1,0, +}_p.
\end{cases}
\end{eqnarray}
 We now cut-off $\mau_p^{1,0,+}=(u_p^{1,0,+},v_p^{1,0,+})$ to obtain the first boundary layer corrector $\mau_p^{1,+}=(u_p^{1,+},v_p^{1,+})$ near $y=2$:
\begin{align} \label{cut.off.2}
u^{1,+}_p = \chi(\frac{\sqrt{\eps}Y_+}{a_0}) u^{1,0, +}_p - \frac{\sqrt{\eps}}{a_0} \chi'(\frac{\sqrt{\eps}Y_+}{a_0}) \int_0^x v_p^{1,0,+}, \hspace{2 mm} v^{1,+}_p := \chi(\frac{\sqrt{\eps}Y_+}{a_0}) v^{1,0, +}_p,
\end{align}
where $\chi$ is the cut-off function defined in (\ref{cutoff}), $a_0>0$ is a constant small enough.
After cutting off (\ref{cut.off.2}), we have the  contribution with $O(\va^{\f12})$ order to the next layer
\begin{eqnarray}
\mathcal{C}_{cut}^{1,+}=&&\f \mu{a_0}\eps^{\f12}\chi'v_p^{1,0,+}+3\f1{a_0}\eps^{\f12}\chi'\partial_{Y_+}u_p^{1,0,+}
+3\f1{a_0^2}\eps\chi''u_p^{1,0,+}\nonumber\\
&&-\f1{a_0^3}\eps^{\f32}\chi'''\int_0^xv_p^{1,0,+}(s,Y_+)ds.
\end{eqnarray}

We also have higher order terms that contribute to the error:
\begin{eqnarray}
\mathcal{C}^{1,+}_{quad} := &&\eps^2 (u^1_e + u^{1,+}_p) \p_x u^{1,+}_p + \eps^2 u^{1,+}_p \p_x u^1_e + \eps^{\frac 32} v^1_e u^{1,+}_{pY_+}  \nonumber\\
&&+ \eps^{\frac 52} v^{1,+}_p u^1_{ey}+ \eps^2 v^{1,+}_p u^{1,+}_{pY_+} - \eps^2 u^{1,+}_{pxx}.
\end{eqnarray}

Besides, due to the approximation of $\mu(y)$ by $\mu(2)=2$ in the support of the cut-off function $\chi(\frac{\eps^{\frac 12} Y_+}{a_0})$, we have another contribution with $O(\va^{\f12})$ order to the error defined by
\begin{align}
\mathcal{C}_{approx}^{1£¬+} := &  (2- \mu(y) )[\chi(\frac{2-y}{a_0}) \p_x u^{1,0,+}_p+\f1{a_0}\eps^{\f12}\chi'v_p^{1,0,+}].
\end{align}
For the higher order terms in the second momentum equation, we will use our auxiliary pressure to move it to the top equation. This is achieved by defining the first auxiliary pressure, $P^{1,a,+}_p$ to zero out the terms contributed from
\begin{eqnarray*}
&&\eps^{\frac 32} (\mu + \eps u^1_e) v^{1,+}_{px} + \eps u^{1,+}_p (\eps v^1_{ex} + \eps^{\frac 32} v^{1,+}_{px}) + \eps^2 v^1_e v^{1,+}_{pY_+} + \eps^{\frac 52} v^{1,+}_p v^{1,+}_{pY_+} + \eps^{\frac 52} v^{1,+}_p v^1_{ey} \\
&&- \eps^{\frac 32} v^{1,+}_{pY_+Y_+} - \eps^{\frac 52} v^{1,+}_{pxx} + \eps^{\frac 32} P^{1,a,+}_{pY_+} =0,
\end{eqnarray*}
which therefore motivates our definition of
\begin{align}  \n
- \eps^{\frac 32} P^{1,a,+}_{P} := & \int_{Y_+}^{\infty} \Big( \eps^{\frac 32} (\mu + \eps u^1_e) v^{1,+}_{px} + \eps u^{1,+}_p (\eps v^1_{ex} + \eps^{\frac 32} v^{1,+}_{px}) + \eps^2 v^1_e v^{1,+}_{pY_+}   \\
&+ \eps^{\frac 52} v^{1,+}_p v^{1,+}_{pY_+}+ \eps^{\frac 52} v^{1,+}_p v^1_{ey} - \eps^{\frac 32} v^{1,+}_{pY_+Y_+} - \eps^{\frac 52} v^{1,+}_{pxx} \Big) \ud Y'.
\end{align}
As a result, we can define the forcing for the next order weak boundary layer via
\begin{align}
F^{2,+} := \eps^{- \frac 32} \Big( -\eps^{\frac 32} \mathcal{C}^{1,+}_{cut} +  \eps^{\frac 32} \mathcal{C}^{1,+}_{approx} + \mathcal{C}^{1,+}_{quad} - \eps^{2} \p_x P^{1,a,+}_P \Big).
\end{align}

 Computing $F^{i,+}$ when  $i=3,...M$ are almost in the same manner. Since we will put the terms with $H^2$ norm less than  $\va^{\f74+\gamma}$  into the remainders $(\mcf_u,\mcf_v)$, we only need the auxiliary pressure $P_p^{i,a,+}$ for $i=1,2,3$. When $i>3$, we will take $P_p^{i,a,+}=0$. Please also notify that the interaction terms containing $u_e^{i,-}u_p^{j,+}$  in $F^{i,+}$ with different scale would not cause any trouble here. In fact, we only have $\va^{\f73}u_e^{2,-}u_p^{1,+}$ in $\mathcal{C}^{2,+}_{quad}$ and $\va^{\f73}v_e^{2,-}v_p^{1,+}$ in $P^{2,a,+}_P$. Both of the two terms containing a boundary layer corrector near $y=2$, which would decay rapidly when $Y_+\rightarrow\infty$. Other interaction terms with form $u_e^{i,-}u_p^{j,+}$ will be put into the remainders $(\mcf_u,\mcf_v)$

 We thus derive the general equation that we study for the boundary layer correctors when $1 \le i \le M - 1$ (that is, all but the last layer):
\begin{align*}
&2 \p_x u^{i,0, +}_p - \p_{Y_+Y_+}u^{i,0, +}_p = F^{i,+}, \\
&u^{i,0, +}_p|_{x = 0} =0 , \hspace{3 mm} u^{i,0, +}_p|_{Y_+ = 0} = -u^{i,+}_e|_{y = 2},\  u^{i,0, +}_p|_{Y_+ \rightarrow \infty} = 0 \\
&v^{i,0, +}_p = \int_{Y_+}^\infty \p_x u^{i,0, +}_p.
\end{align*}
Let us therefore consider the abstract problem (dropping indices):
\begin{subequations}
\begin{align} \label{A.8}
&2 \p_x u  - \p_{Y}^2 u = F, \qquad (x, Y) \in (0, L) \times (0, \infty) \\
&v := \int_Y^\infty \p_x u \ud Y', \\ \label{A.9}
&u|_{x = 0} = 0, \qquad u|_{Y = 0} = g(x), \qquad u|_{Y \rightarrow \infty} = 0.
\end{align}
\end{subequations}
For this abstract problem, we have the following Lemma:

\begin{lemma} \label{A.lemma3} Assume that $F(x,Y)$ decays rapidly at infinity, i.e., for any $m\geq0$, there is a constant $M>0$,
such that
\begin{align*}
\|(1+Y)^m \p_x^n \p_Y^{l}F \| \leq M \text{ for } 0 \le 2n+ l \le K,
\end{align*}
where $K>0$ is a constant sufficiently large. Then there exists a unique solution $(u, v)$ to \eqref{A.8}--\eqref{A.9}
that satisfies the following estimate:
\begin{equation}
\|(1+Y)^m \partial_x^n \partial_Y^{l} \{ u, v \}\| \leq C(m,n,l)(M+\|g\|_{H^{K+1}}) \text{ for any } 2n + l \leq K+2, \label{2.2.2}
\end{equation}
where the constant $C$ does not depend on $Y$.
\end{lemma}

\begin{proof}  First, a standard homogenization enables us to consider the Dirichlet problem, $g = 0$, up to modifying the forcing $F$. Indeed, fixing a cut-off function $\chi(Y)$ so that $\chi(0) = 1, \chi(Y) = 0$, when $Y\geq2$, and $ \int_0^{\infty} \chi(Y') \ud Y' = 0$, we can consider the unknowns
\begin{align}
\tilde{u} := u - \chi(Y) g(x),\nonumber
\end{align}
which will satisfy the system
\begin{eqnarray}
&&2\p_x \tilde{u}  - \p_{Y}^2 \tilde{u} = \tilde{F}, \qquad (x, Y) \in (0, L) \times (0, \infty) \label{A.10}\\
&&\tilde{u}|_{x = 0} = 0, \qquad \tilde{u}|_{Y = 0} = 0, \qquad \tilde{u}|_{Y \rightarrow \infty} = 0.
\end{eqnarray}
Above, the modified forcing
\begin{align}
\tilde{F} := F -2\chi(Y) g'(x) +  g(x)\chi''(Y) .
\end{align}
We now drop the $\tilde{u}, \tilde{v}$ notation, and simply consider the homogenized problem above.

First we multiply (\ref{A.10}) by $u_x $ to obtain
\begin{align}
\frac{\p_x}{2} \int_0^\infty u_Y^2 \ud Y + \int_0^{\infty}  |\p_x u|^2 \ud Y  =  \int_0^{\infty} F u_x \ud Y.
\end{align}
Consequently, we have
\begin{eqnarray}
&&\int_\Omega u_x^2dYdx+\sup_{x\in[0,L]}\int_0^\infty u_Y^2dY\leq C\|F\|_{L^2(\Omega)}^2.
\end{eqnarray}
 From equation (\ref{A.10}), we have $u_x|_{x=0}=\f12 F(0,Y)$. Similarly, we can obtain from equation (\ref{A.10})
 $$|\partial_{x}^nu(0,\cdot)|_{L^2}\leq C(n)\|F\|_{H^{n+1}(\Omega)}.$$
 for any $n\geq0$.

 Now we take the derivative of equation (\ref{A.10}) with respect to $x$ for $n$ times.  Repeating the above process and using the resulting equation, we have
\begin{eqnarray}
\int_\Omega (\partial_x^nu_x)^2dYdx+\int_\Omega (\partial_x^{n}u_{YY})^2dYdx+\sup_{x\in[0,L]}\int_0^\infty (\partial_x^nu_Y)^2dY\leq C\|F\|_{H^{n+1}}^2.\nonumber
\end{eqnarray}
 To get weighted  estimates, we  multiply equation (\ref{A.10}) with $u_x (1+Y)^{2m}$ to have
 \begin{eqnarray}
 &&2\int_0^\infty(1+Y)^{2m}u_x^2dy+\f12\f{d}{dx}\int_0^\infty (1+Y)^{2m}(u_{Y})^2dy\nonumber\\
 =&&\int_0^\infty(1+Y)^{2m}u_xFdy-2m\int_0^\infty (1+Y)^{2m-1}u_{x}u_{Y}dy\nonumber\\
\leq &&C\|(1+Y)^mu_x(x,\cdot)\|_{L^2_y}\{\|(1+Y)^mF(x,\cdot)\|_{L^2_y}+\|(1+Y)^mu_Y(x,\cdot)\|_{L^2_y}\}.\nonumber
 \end{eqnarray}
 Consequently
 \begin{eqnarray}
 \|(1+Y)^mu_x\|_{L^2(\Omega)}+\sup_{x\in[0,L]}\|(1+Y)^mu_Y(x,\cdot)\|_{L^2_y}
\leq C(m)\|(1+Y)^mF\|_{L^2}.
 \end{eqnarray}
 Similarly, we can obtain
  \begin{eqnarray}
&& \|(1+Y)^m\partial_x^nu_x\|_{L^2(\Omega)}+\|(1+Y)^m\partial_x^nu_{YY}\|_{L^2(\Omega)}+\sup_{x\in[0,L]}\|(1+Y)^m\partial_x^nu_Y(x,\cdot)\|_{L^2_y}\nonumber\\
&&\leq C(m,n)(M+\|g\|_{H^{K+1}}).
 \end{eqnarray}
Finally, for any $m,n,l\geq0$, we use the differential operator $\partial_x^n\partial_Y^{2l}$ to act on equation (\ref{A.10}), and then  multiply the resulting equation with $(1+Y)^{2m}\partial_x^n\partial_Y^{2l}u_x$,  to get
\begin{eqnarray}
&& \|(1+Y)^m\partial_x^n\partial_Y^{2l}u_x\|_{L^2(\Omega)}+\|(1+Y)^m\partial_x^n\partial_Y^{2l+2}u\|_{L^2(\Omega)}\nonumber\\
&&+\sup_{x\in[0,L]}\|(1+Y)^m\partial_x^n\partial_Y^{2l+1}u(x,\cdot)\|_{L^2_y}\nonumber\\
\leq&& C(m,n,l)(M+\|g\|_{H^{K+1}}).
 \end{eqnarray}
 Then (\ref{2.2.2}) follows immediately.
\end{proof}

For $i = M$, we consider now
\begin{subequations}
\begin{align} \label{A.11}
&2\p_x u  - \p_{Y}^2 u = F, \qquad (x, Y) \in (0, L) \times (0, \infty) \\
&v := - \int_0^Y \p_x u \ud Y', \\ \label{A.12}
&u|_{x = 0} = 0, \qquad u|_{Y = 0} = g(x), \qquad \partial_Yu|_{Y \rightarrow \infty} = 0.
\end{align}
\end{subequations}
Compared to \eqref{A.8} -- \eqref{A.9}, the main difference here is $v(x,0)=0$. We would have the following lemma
\begin{lemma}\label{A.lemma4}
Assume  that $F(x,Y)$ decays rapidly at infinity,  i.e., for any $m\geq0$, there is a constant $M>0$,
such that
\begin{align*}
\|(1+Y)^m \p_x^n \p_Y^{l}F \| \leq M\ \text{ for } 0 \le 2n+ l \le K.
\end{align*}
Then, there exists a unique solution $(u, v)$ of \eqref{A.11}--\eqref{A.12}, satisfying the following estimates:
\begin{align*}
&\|(1+Y)^m \p_x^n \p_Y^{l} u\| \le  C(m,n,l)(M+\|g\|_{H^{K+1}})\ \text{ for any } 0 \le 2n + l \le K+2, \\
&\|(1+Y)^m \p_x^n \p_Y^{l} v_Y\| \le  C(m,n,l)(M+\|g\|_{H^{K+1}})\ \text{ for any } 0 \le 2n + l \le K+1, \\
&\|\p_x^n \{  \frac{v}{Y} \}\| \le   C(m,n,l)(M+\|g\|_{H^{K+1}})\ \text{ for any } 0 \le 2n \le K+2.
\end{align*}
\end{lemma}
The proof of Lemma \ref{A.lemma4} would be almost the same as Lemma \ref{A.lemma3}
, and we omit the details here.

\subsection{Proof of Theorem \ref{thm2.2}}

Combing Lemma \ref{lem2.2}-Lemma \ref{A.lemma4}, Lemma \ref{A.lemma1}, Lemma \ref{A.lemma2}, we have the following Lemma:

\begin{lemma}\label{th.asy.1} Let $u^{i,\pm}_p, v^{i,\pm}_p$ be solutions obtained in the above process,  then we have for any $m \ge 0$,
\begin{align*}
&\|(1+Y)^m \p_x^n \p_Y^l \{ u^{i,\pm}_p, v^{i,\pm}_p \}\| \leq  C_{m,n,l}\alpha_0\text{ for } 1 \le i \le M - 1 \\
&\|(1+Y)^m \p_x^n \p_Y^l  u^{M,\pm}_p\| \leq  C_{m,n,l}\alpha_0 \text{ for } l \ge 1, \\
&\| \p_x^n \{ u^{M,\pm}_p,  \frac{v^{M,\pm}_p}{Y} \}\| \leq  C_{m,n,l}\alpha_0,
\end{align*}
where $n,l$ are integers  satisfying $2n + l \leq k+2$, $(u_p^{M,\pm},v_p^{M,\pm})$ are the last boundary layer correctors.
\end{lemma}
\bigskip
\noindent $\mathbf{Proof\  of \ Theorem \ \ref{thm2.2}}$:

First of all, the divergence free condition of $\mau_s$ and  the boundary conditions in (\ref{2.2.1}) follow directly from the process of construction of $\mau_p^{i,\pm}$.

Then recalling the definition of $Y_\pm$ and the expression of $\mau_s$ in (\ref{asy:1:1:1}), estimates (\ref{2.2.3})-(\ref{2.2.5}) follows immediately from Lemma \ref{th.asy.1}.

Next, using definition (\ref{remain1})-(\ref{remain2}), $\mathcal{F}_u, \mathcal{F}_v$ can be divided into two parts, i.e.
\begin{align*}
\va^{M_0}\mathcal{F}_u = \mathcal{C}_{Euler,u}+\mathcal{C}_{prandtl,u}, \\ \n
\va^{M_0}\mathcal{F}_v =  \mathcal{C}_{Euler, v}+\mathcal{C}_{prandtl,v},
\end{align*}
where $\mathcal{C}_{Euler,u}, \mathcal{C}_{Euler,v}$ are defined in (\ref{euler.2})-(\ref{euler.3}).
By Lemma \ref{lem2.2}, the Euler contributions to the error in \eqref{euler.2} and \eqref{euler.3} can be written as
\begin{equation}\label{asy.euler.0}
\begin{cases}
\mathcal{C}_{Euler,u}=\eps^2 (u^1_e \p_x u^1_e + v^1_e \p_y u^1_e - \Delta u^1_e) +O(\va^{\f83}), \\
\mathcal{C}_{Euler,v} = \eps^2 (u^1_e \p_x v^1_e + v^1_e \p_y v^1_e - \Delta v^1_e) +O(\va^{\f83}),
\end{cases}\end{equation}
while the weak boundary layer contributions to the error can be arbitrarily small if we expanding enough sub-layers. Here $M=10$ would be sufficient.
If we define
$$\mathcal{T}_{u,\va^2}=\mathcal{C}_{Euler,u},\ \mathcal{T}_{v,\va^2}=\mathcal{C}_{Euler,v},\ \mathcal{F}_{u,\va^3}=\mathcal{C}_{prandtl,u},\ \mathcal{F}_{v,\va^3}=\mathcal{C}_{prandtl,v},$$
then (\ref{2.2.8}) follows immediately by Lemma \ref{lem2.2} and Lemma \ref{th.asy.1} and we finish the proof of Theorem \ref{thm2.2}.

\bigskip

\renewcommand{\theequation}{\thesection.\arabic{equation}}
\setcounter{equation}{0}
\section{Existence and uniform-in-$\va$ estimates of solutions to the linearized system}

In this section, we will study the existence as well as the uniform-in-$\va$ estimates of the solutions to the  linearized system of (\ref{main.NS}).

For given $(\bar{u},\bar{v})$ and $\mathbf{F}=(F_1,F_2)$, we  consider the linear system
\begin{eqnarray}
\dv\mau=0,\label{linear1}\\
u_su_x+u_{sy}v+u_{sx}u+v_su_y-\va\Delta u+\partial_xP=N_1(\bar{u},\bar{v})+F_1 \label{linear2}\\
u_sv_x+v_sv_y+v_{sx}u+vv_{sy}-\va\Delta v+\partial_yP=N_2(\bar{u},\bar{v})+F_2,\label{linear3}
\end{eqnarray}
with boundary condition:
\begin{eqnarray}\label{boundary2}
&&\ v_x|_{x=0}=0;\ u|_{x=0}=0;\ v|_{x=L}=0;\ v_{xx}|_{x=L}=0;\nonumber\\
&&u|_{y=0}=v|_{y=0}=v_y|_{y=0}=u|_{y=2}=v|_{y=2}=v_y|_{y=2}=0,
\end{eqnarray}
where $u_s, v_s$ are known functions defined in (\ref{0.exp}).

First we differentiate (\ref{linear2}) with $y$ and (\ref{linear3}) with $x$, using the divergence free condition to obtain the following equation for $\curl \mau=u_y-v_x$:
\begin{equation}
-u_s v_{yy}+u_{syy}v-u_sv_{xx}+u_{sxx}v+S(u,v)-\va\Delta(u_y-v_x)=\curl \mathbf{N}+\curl \mathbf{F},\label{curl1}
\end{equation}
where $S(u,v)=\partial_y(u_{sx}u+v_su_y)-\partial_x(v_sv_y+v_{sx}u)$.

Recalling that $q=\f{v}{u_s},$  equation (\ref{curl1}) can be written as
 \begin{eqnarray}
-\partial_y(u_s^2q_y)-\partial_x(u_s^2q_x)+S(u,v)-\va\Delta(u_y-v_x)=\curl \mathbf{N}+\curl  \mathbf{F}.\label{curl.2}
\end{eqnarray}

\subsection{Existence of solutions to the linearized system}

In this subsection, we will prove the existence of solutions to the linear system (\ref{linear5})-(\ref{linear7}) with boundary condition (\ref{boundary2}). Due to the boundary condition $v_{xx}=0$ on $x=L$, we can not obtain the existence of the solutions  directly from the second order equations (\ref{linear5})-(\ref{linear7}). Instead, by introducing the stream function, we will first prove the existence of solutions to a fourth order equation for the stream function. Then by an elliptic equation of second order for the pressure, we can prove the  existence of solutions to the second order linear system (\ref{linear5})-(\ref{linear7}).

 For given $\mathbf{f}=(f_1,f_2)$,  we will first consider the following linear system:
\begin{eqnarray}
\dv\mau=0,\label{linear5}\\
u_su_x+u_{sy}v+u_{sx}u+v_su_y-\va\Delta u+\partial_xP=f_1,\label{linear6}\\
u_sv_x+v_sv_y+v_{sx}u+vv_{sy}-\va\Delta v+\partial_yP=f_2,\label{linear7}
\end{eqnarray}
with boundary condition (\ref{boundary2}).
The main theorem for this subsection reads as
\begin{theorem}\label{th.linear2}
 For any given $\mathbf{f}=(f_1,f_2)\in H^1(\Omega)\times H^1(\Omega)$, there exists a unique solution $(u,v,\nabla P)$ to the system (\ref{linear5})--(\ref{linear7}) satisfying the boundary condition (\ref{boundary2}) and the following estimate:
 \begin{equation}
 \|u\|_{H^3}+\|v\|_{H^3}+\|\nabla P\|_{H^1}\leq C(\va)\|\mathbf{f}\|_{H^1}.\label{3.1.0}
 \end{equation}
\end{theorem}
To prove Theorem \ref{th.linear2}, we first consider the following equation for $\curl \mau$:
\begin{equation}
-u_s v_{yy}+u_{syy}v-u_sv_{xx}+u_{sxx}v+S(u,v)-\va\Delta(\curl \mau)=\curl \mathbf{\mathbf{f}}.\label{curl2}
\end{equation}
Then we introduce the stream-function $\psi$ of $(u,v)$:
\begin{equation}
\psi_x=-v,\ \psi_y=u.
\end{equation}
If we assume that $\psi(0,0)=0$, then $\psi$ satisfies the following system
\begin{eqnarray}
u_s \psi_{xyy}+u_s\psi_{xxx}-\Delta u_{s}\psi_x+S(\psi_y,-\psi_x)-\va\Delta^2\psi&=&\curl \mathbf{\mathbf{f}}\ \ \  \text{in}\ \Omega;\label{3.1.1}\\
\ \psi_{xx}|_{x=0}=\psi|_{x=0}=0;\ \psi_x|_{x=L}=0;\ \psi_{xxx}|_{x=L}&=&0;\label{3.1.9}\\
\psi|_{y=0}=\psi_y|_{y=0}=\psi|_{y=2}=\psi_y|_{y=2}&=&0.\label{3.1.2}
\end{eqnarray}

To prove the existence of solutions to system (\ref{3.1.1})-(\ref{3.1.2}), we first consider the following biharmonic equation
\begin{equation}
\Delta^2\psi=f,\ \ \text{in}\ \Omega,\label{linear8}
\end{equation}
with boundary condition (\ref{3.1.9})-(\ref{3.1.2}).

We define the function space via:
$$H^m_0=\{v\in H^m(\Omega)\;|\; (\ref{3.1.9})- (\ref{3.1.2})\mbox{ are satisfied}\}.$$
 For any $\psi,\phi\in H^2_0(\Omega)$, we  define the bilinear form
\begin{equation*}
B[\psi,\phi]=\int_\Omega(\psi_{xx}\phi_{xx}+2\psi_{xy}\phi_{xy}+\psi_{yy}\phi_{yy})dxdy.
\end{equation*}
\begin{definition}
We say that $\psi\in H^2_0$ is a weak solution to (\ref{linear8}) with boundary conditions (\ref{3.1.9})-(\ref{3.1.2}) if
\begin{equation}
B[\psi,\phi]=(f,\phi) \ \ \text{for all}\ \phi\in H^2_0. \label{3.1.3}
\end{equation}
\end{definition}
\begin{lemma}\label{lem1}
For any $f\in L^2(\Omega)$, there exists a unique solution $\psi\in H^4_0(\Omega)$ to equation (\ref{linear8}) satisfying boundary condition (\ref{3.1.9})-(\ref{3.1.2}), and the following estimate:
\begin{equation}
\|\psi\|_{H^4}\leq C\|f\|_{L^2},
\end{equation}
where $C$ is a constant depending on $\Omega$.
\end{lemma}
\begin{proof}

First, It is easy to  prove the existence of weak solution by Lax-Milgram Theorem.
Obviously, there exist constants $C_1,C_2>0$ such that for any $\psi,\phi\in H^2_0(\Omega)$,
$$B[\psi,\phi]\leq C_1\|\psi\|_{H_2}\|\phi\|_{H_2},$$
and
$$B[\psi,\psi]\geq C_2\|\psi\|_{H_2}^2.$$
So $B$ is a coercive bilinear operator. The Lax-Milgram Theorem implies that there exists a unique weak solution $\psi\in H^2_0(\Omega)$  to  system (\ref{3.1.9})- (\ref{linear8}) satisfying (\ref{3.1.3})  and
\begin{equation}
\|\psi\|_{H^2}\leq C\|f\|_{L^2},\label{3.1.8}
\end{equation}
Moreover, since $f\in L^2(\Omega)$, we have (\ref{linear8}) holds a.e. in $\Omega$.

Next, to obtain the global $H^4$ regularity of $\psi$ in $\Omega$,   we  multiply (\ref{linear8}) with $\psi_{xx}$. Using the boundary condition (\ref{3.1.9})-(\ref{3.1.2}) and estimate (\ref{3.1.8}), we have
\begin{equation}
\int_\Omega [\psi_{xxx}^2+2\psi_{xxy}^2+\psi_{xyy}^2]dxdy=\int_{\Omega}f\psi_{xx}dxdy\leq C\|f\|^2.\label{linear3.1.1}
\end{equation}
Then we make the odd extension of $\psi$ and $f$ with respect to $x=0$,  and denote by
\begin{eqnarray*}
\psi_{odd}(x,y)=\begin{cases}\psi(x,y)~\quad\text{if}~x>0,\\
-\psi(-x,y)\quad\text{if}~x<0, \end{cases}
f_{odd}(x,y)=\begin{cases}f(x,y)~\quad\text{if}~x>0,\\
-f(-x,y)\quad\text{if}~x<0.\end{cases}
\end{eqnarray*}
By (\ref{linear3.1.1})  and the boundary condition on $x=0$, we have $\partial_x\psi_{odd}\in H^2(\bar\Omega)$ and
\begin{equation}
\int_{\bar{\Omega}} [(\partial_{xxx}\psi_{odd})^2+2(\partial_{xxy}\psi_{odd})^2+(\partial_{xyy}\psi_{odd})^2]dxdy\leq C\|f\|^2_{\Omega},\label{linear9}
\end{equation}
 here $\bar{\Omega}=[-L,L]\times[0,2]$.

Next we define a smooth cutoff function $0\leq \eta(x)\leq1$ satisfying
\begin{eqnarray*}
\eta(x)=\begin{cases}1~\quad\text{if}~-\f34L\leq x\leq\f34L,\\
0\quad\text{if}~-L\leq x\leq-\f45L,\ \text{or}\  \f45L\leq x\leq L.\end{cases}
\end{eqnarray*}
If we denote $\tilde\psi=\eta(x)\psi_{odd}$ and $\tilde\Omega\subset\bar\Omega$ is a smooth domain containing $[-\f45L,\f45L]\times[0,2]$,
then it is easy to check that $\tilde\psi$ is the weak solution to the following bihamornic system with Dirichlet boundary condition in
$\tilde\Omega$:
\begin{eqnarray*}\begin{cases}
\Delta^2\tilde\psi= \tilde f,\ \text{in}\ \tilde\Omega\\
\tilde\psi=\partial_n\tilde\psi=0,\ \text{on}\ \partial\tilde\Omega ,\end{cases}
\end{eqnarray*}
where $\tilde f=\eta'\partial_{xxx}\psi_{odd}+\eta''\partial_{xx}\psi_{odd}+\eta'''\partial_{x}\psi_{odd}
+\eta''''\psi_{odd}+\eta'\partial_{xyy}\psi_{odd}+\eta''\partial_{yy}\psi_{odd}$.

Using (\ref{3.1.8})  and (\ref{linear9}), we have $$\|\tilde f\|_{L^2}\leq C\|f\|_{L^2}.$$ Then by Theorem 2.20 in \cite{FGS},
we have $\tilde\psi\in H^4(\tilde\Omega)$ satisfying
\begin{equation*}
\|\tilde \psi\|_{H^4(\tilde{\Omega})}\leq C\|f\|_{L^2(\tilde\Omega)}.
\end{equation*}
Similarly we can  make the even extension of $\psi$ and $f$ with respect to $x=L$ and repeat the process above to obtain
\begin{equation*}
\| \psi\|_{H^4(\Omega)}\leq C\|f\|_{L^2(\Omega)}.
\end{equation*}
Thus, we finish the proof of the Lemma.
\end{proof}

Next we turn to the following system
\begin{eqnarray}
u_s \psi_{xyy}+u_s\psi_{xxx}-\Delta u_{s}\psi_x+S(\psi_y,-\psi_x)-\va\Delta^2\psi&=&f\ \ \  \text{in}\ \Omega;\label{3.1.5}\\
\ \psi_{xx}|_{x=0}=\psi|_{x=0}=0;\ \psi_x|_{x=L}=0;\ \psi_{xxx}|_{x=L}=0;\\
\psi|_{y=0}=\psi_y|_{y=0}=\psi|_{y=2}=\psi_y|_{y=2}=0.\label{3.1.6}
\end{eqnarray}
By Lemma \ref{lem1}  and Leray-Schauder fixed point theory, which can be found in \cite[Chapter 11]{G-T}, we have the following lemma:
\begin{lemma}
For any $f\in L^2(\Omega)$, there exists a unique solution $\psi\in H^4(\Omega)$ of the boundary value problem (\ref{3.1.5})-(\ref{3.1.6})
satisfying
\begin{eqnarray}
\|\psi\|_{H^4}\leq C(\va)\|f\|_{L^2}.\label{linear.4}
\end{eqnarray}
\end{lemma}

\begin{proof}
First, taking $\mathcal{B}=H^3_0(\Omega)$, for any $\phi\in \mathcal{B},\ t\in[0,1]$, we define an operator $\mathrm{T}:\mathcal{B}\times[0,1]\rightarrow \mathcal{B}$ by the solution of the following system:
\begin{eqnarray}
-\va\Delta^2\psi=t\{f-[u_s \phi_{xyy}-\Delta u_{s}\phi_x+u_s\phi_{xxx}+S(\phi_y,-\phi_x)]\},\label{linear.2}\\
\ \psi_{xx}|_{x=0}=\psi|_{x=0}=0;\ \psi_x|_{x=L}=0;\ \psi_{xxx}|_{x=L}=0;\\
\psi|_{y=0}=\psi_y|_{y=0}=\psi|_{y=2}=\psi_y|_{y=2}=0.
\end{eqnarray}
Then by Lemma \ref{lem1}, $T$ is a compact operator and $T(\phi,0)=0$ for all $\phi\in \mathcal{B}$.

Next, remember that $\psi_x=-v, \ \psi_y=u,\ q=\f{v}{u_s}$,  if we replace $\phi$ by $\psi$,  then  (\ref{linear.2}) can be written as
\begin{equation}
-t[\partial_y(u_s^2q_y)-\partial_x(u_s^2q_x)+S(\psi_y,-\psi_x)]-\va\Delta(u_y-v_x)=tf.\label{linear.3}
\end{equation}
Multiplying (\ref{linear.3}) with $tv$ to obtain
\begin{eqnarray}
&&-\int_\Omega t^2v[\partial_y(u_s^2q_y)-\partial_x(u_s^2q_x)+S(\psi_y,-\psi_x)]dxdy-\va\int_\Omega\Delta(u_y-v_x)tvdxdy\nonumber\\
=&&\int_\Omega tvfdxdy.\label{linear3.1.3}
\end{eqnarray}
By similar  process in Lemma \ref{lem.2}, we have
\begin{eqnarray}
&&\int_\Omega t^2[ -u_s v_{yy}+u_{syy}v-u_sv_{xx}+u_{sxx}v]vdxdy\nonumber\\
=&&t^2\int_\Omega u_s(v_y^2+v_x^2)dxdy+t^2\int_\Omega(\f12 u_{syy}v^2+u_{sxx}v^2+ u_{sx}vv_x)dxdy\nonumber\\
\gtrsim&&t^2\|\sqrt{u_s} v_{y}\|^2+t^2\|\sqrt{u_s} v_{x}\|^2,
\end{eqnarray}
and
\begin{eqnarray}
&&t^2|\int_\Omega S(u,v)vdxdy|=t^2|\int_\Omega [-\partial_y(u_{sx}u+v_su_y)+\partial_x(v_sv_y+v_{sx}u)]vdxdy|\nonumber\\
=&&t^2|\int_\Omega (u_{sx}u+v_su_y)v_ydxdy+\int_\Omega (v_sv_{xy}+v_{sxx}u)vdxdy|\nonumber\\
\lesssim &&t^2\{L\va^{\f12}\|\sqrt{u_s}v_y\|^2+L\va\|\sqrt{u_s}v_y\|\|v_{yy}\|+\|v\|_{L^2}(\va\|v_{xy}\|+L\va\|\sqrt{u_s}v_y\|)\}\nonumber\\
\lesssim &&t^2L^{\f14}(\|\sqrt{u_s} v_{y}\|^2+\|\sqrt{u_s} v_{x}\|^2)+\va^2\|\sqrt{u_s}\nabla^2q\|^2,
\end{eqnarray}
while
\begin{eqnarray}
&&-\va\int_\Omega\Delta (u_y-v_x)tvdxdy\nonumber\\
=&&\f12t\va\int_0^2 u_y^2|_{x=L}+t\va\int_0^2 v_y^2|_{x=0}
-t\va\int_0^2 v_{xx}v|_{x=0}-\f12t\va\int_0^2 v_x^2|_{x=L}\nonumber\\
\gtrsim&&t\va|u_y(L,\cdot)|^2+t\va|v_y(0,\cdot)|^2-t\va|u_s^{-\f14}v_{xx}(0,\cdot)||u_s^{\f14}v(0,\cdot)|-t\va\|\sqrt{u_s} v_{x}\|\|\sqrt{u_s}\nabla^2q\|,\nonumber
\end{eqnarray}
where
\begin{eqnarray}
&&t^2\va^2|u_s^{\f14}v(0,\cdot)|^2|u_s^{-\f14}v_{xx}(0,\cdot)|^2\nonumber\\
=&&t^2\va^2\{\int_\Omega(2u_s^{\f12}vv_x+\f12u_s^{-\f12}u_{sx}v^2)dxdy\}\{\int_\Omega[2u_s^{-\f12}v_{xx}v_{xxx}
-\f12u_s^{-\f32}u_{sx}v_{xx}^2]dxdy\} \nonumber\\
\leq &&t^2\va^2(\|\sqrt{u_s}v_{x}\|\|v\|+\va^{\f14}\|v\|^2)[(\|\sqrt{u_s}q_{xx}\|
+\|q_{x}\|)\|v_{xxx}\|+\va^{\f12}(\|\sqrt{u_s}q_{xx}\|+\|q_{x}\|)^2]\nonumber\\
\lesssim&&L^{\f12}t^2\|\sqrt{u_s} \nabla v\|^2(\va\|\sqrt{u_s}\nabla^2q\|+\va^3\|v_{xxx}\|^2).\label{linear3.1.4}
\end{eqnarray}
Combing (\ref{linear3.1.3})-(\ref{linear3.1.4}) and (\ref{linear.2.1}), we have
\begin{eqnarray}
t^2\|\sqrt{u_s} v_{y}\|^2+t^2\|\sqrt{u_s} v_{x}\|^2\leq C[L^{\f14}(\va\|\sqrt{u_s}\nabla^2q\|+\va^3\|v_{xxx}\|^2)+\|f\|^2],\label{linear3.1.8}
\end{eqnarray}
where the constant $C$ does not depend on $t,\;\va$.

Then, we multiply (\ref{linear.3}) with $q_x$ to get
\begin{eqnarray}
&&\int_\Omega[ -t\partial_y(u_s^2\partial_yq)-t\partial_x(u_s^2q_x)+tS(u,v)-\va\Delta(u_y-v_x)]q_xdxdy=\int_\Omega fq_xdxdy.\nonumber\\\label{linear3.1.5}
\end{eqnarray}
Similar to the process in Lemma \ref{lem3.1}, we obtain
\begin{eqnarray}
&&\int_\Omega tq_x[-\partial_y(u_s^2\partial_yq)-\partial_x(u_s^2q_x)]dxdy\nonumber\\
=&&-\f12t\int_\Omega u_s^2q_y^2|_{x=0}dy-\f12t\int_\Omega u_s^2q_x^2|_{x=L}dy+\f12t\int_\Omega u_{sx}^2q^2|_{x=0}dy\nonumber\\
&&-t\int_\Omega u_su_{sx}(q_x^2+q_y^2)dxdy\nonumber\\
\lesssim&&-t(|u_sq_y(0,\cdot)|^2+|u_sq_x(L,\cdot)|^2)+Lt\va\int_\Omega (u_sq_{xx}^2+u_sq_{xy}^2)dxdy,
\end{eqnarray}
and
\begin{eqnarray}
&&t|\int_\Omega S(u,v)q_xdxdy|\nonumber\\
=&&t|\int_\Omega [-\partial_y(u_{sx}u+v_su_y)+\partial_x(v_sv_y+v_{sx}u)]q_xdxdy|\nonumber\\
=&&t|\int_\Omega[(u_{sx}u+v_su_y)q_{xy}+(v_sv_{xy}+v_{sxx}u)q_x]dxdy|\nonumber\\
\leq&&tL^{\f12}\|\sqrt{u_s}q_{xy}\|(\va^{\f12}\|\sqrt{u_s}v_y\|+\va\|\sqrt{u_s}\nabla^2q\|),
\end{eqnarray}
while by virtue of (\ref{linear.2.11}),
\begin{eqnarray}
&&-\va\int_\Omega\Delta(u_y-v_x)q_xdxdy
\lesssim -\va\|\sqrt{u_s}\nabla^2q\|^2+\va^{\f14+3}\|v_{xxx}\|^2.\label{linear3.1.6}
\end{eqnarray}
Putting (\ref{linear3.1.5})-(\ref{linear3.1.6}) together, we conclude
\begin{eqnarray}
\va\|\sqrt{u_s}\nabla^2q\|^2\leq C(\va)\|f\|^2+C_0(L^{\f14}t^2\|\sqrt{u_s}\nabla v\|^2+\va^{\f14+3}\|v_{xxx}\|^2), \label{linear3.1.7}
\end{eqnarray}
where the constant $C(\va)$ does not depend on $t$, while $C_0$ does not depend on $t,\va$.

Next, we multiply (\ref{linear.3}) with $\va^2v_{xxx}$ to find that
\begin{eqnarray}
&&-\va^3\int_\Omega v_{xxx}\Delta(u_y-v_x)dxdy=\va^3\int_\Omega(v_{xxx}^2+v_{xxy}^2+v_{xyy}^2)dxdy\nonumber\\
\leq&&\va^2\int_\Omega t|f-u_s v_{yy}+u_{syy}v-u_sv_{xx}+u_{sxx}v+S(u,v)||v_{xxx}|dxdy\nonumber\\
\leq&&\va^2(\|v_{xxx}\|+\|v_{xxy}\|)(\|f\|+\|\sqrt{u_s}\nabla^2q\|),  \nonumber
\end{eqnarray}
which implies
\begin{equation}
\va^3\|\nabla^2v_x\|^2 \leq C(\|f\|^2+\va\|\sqrt{u_s}\nabla^2q\|^2), \label{linear3.1.2}
\end{equation}
where the constant $C$ does not depend on $\va,t$.

Combing (\ref{linear3.1.8}) and (\ref{linear3.1.7}) with (\ref{linear3.1.2}), we obtain
$$\va\|\sqrt{u_s}\nabla^2q\|^2+\va^3\|\nabla^2v_x\|^2 \leq C(\va)\|f\|^2\quad\mbox{for any }0\leq t\leq 1,$$
where the constant $C(\va)$ does not depend on $t$.
 Recalling that $u=\psi_y, v=-\psi_x$, and the divergence free condition of $\mau$, we have
 $$\|t\{f-[u_s \psi_{xyy}-u_{syy}\psi_x+u_s\psi_{xxx}+S(\psi_y,-\psi_x)]\}\|\leq C(\va)\|f\|_{L^2}+\|v_su_{yy}\|.$$

Using Lemma \ref{lem1} again and recalling that $|v_s|_{L^\infty}\leq \va\alpha_0$, we have
$$\va\|\psi\|_{H^4}\lesssim C(\va)\|f\|_{L^2}+\|v_su_{yy}\|\lesssim C(\va)\|f\|_{L^2}+\va\alpha_0\|\psi_{yyy}\|.$$
Taking $\alpha_0$ small enough, then
\begin{equation}\|\psi\|_{H^4}\lesssim C(\va)\|f\|_{L^2}, \label{3.1.10}\end{equation}
where the constant $C(\va)$ does not depend on $t$.

 The Leray-Schauder fixed point theorem implies that there exists a unique $\psi\in\mathcal{B}$ with estimate (\ref{3.1.10}),
 such that $\psi=T(\psi,1)$. Then the proof is completed.
\end{proof}

\noindent $\mathbf{Proof\ of\ Theorem\ \ref{th.linear2}:}$

To prove Theorem \ref{th.linear2}, we also have to show the existence of the pressure $P$. We first differentiate (\ref{linear6})
with respect to $x$ and differentiate (\ref{linear7}) with respect to $y$ to have
\begin{eqnarray}
\Delta P=\partial_xf_1+\partial_yf_2-(2u_{sy}v_x+4v_{sy}v_y+2v_{sx}u_y).\label{div.1}
\end{eqnarray}
Moreover, we assume the following boundary condition for $P$ from equations (\ref{linear6}) and (\ref{linear7}):
\begin{eqnarray}
P_x=f_1-[u_su_x+u_{sy}v+u_{sx}u+v_su_y-\va\Delta u]\ \text{ on}\ \{x=0\}\cup\{x=L\},\label{boundary.3.1}\\
P_y=f_2-[u_sv_x+v_sv_y+v_{sx}u+vv_{sy}-\va\Delta v]\ \text{ on}\ \{y=0\}\cup\{y=2\}.\label{boundary.3.2}
\end{eqnarray}
Then, by the theory for elliptic equations, there exists a unique solution $P$ up to a constant to the system (\ref{div.1})-(\ref{boundary.3.2}) satisfying the estimate
\begin{equation}
\|\nabla P\|_{H^1}\leq C(\va)(\|\mathbf{f}\|_{H^1}+\|u\|_{H^3}+\|v\|_{H^3}) \leq C(\va)\|\mathbf{f}\|_{H^1}.
\end{equation}
Therefore, we have obtained a triple $(u,v,\nabla P)$ satisfying (\ref{3.1.0}). Besides, the equations (\ref{linear6})-(\ref{linear7}) follow immediately
from (\ref{curl2}), (\ref{div.1}) and the boundary conditions (\ref{boundary.3.1})-(\ref{boundary.3.2}).

\subsection{Uniform-in-$\va$ estimates}

In this subsection, we will work on the uniform-in-$\va$ estimates of the linear system (\ref{linear1})-(\ref{linear3}).
The main result of this subsection reads as follows.
\begin{theorem}\label{th.linear1}
For given $\bar{\mathbf{u}}=(\bar{u},\bar{v})\in \mathcal{X}$, $(F_1,F_2)\in H^1(\Omega)\times H^1(\Omega)$,
if $\mau=(u,v)$ is the solution to  system (\ref{linear1})- (\ref{boundary2}), then we have the following estimate:
\begin{equation}
\|\mau\|_{\mathcal{X}}^2\leq C\{\va^\gamma\|\bar{\mau}\|_{\mathcal{X}}^4+\|  \mathbf{F}\|_{H_1}^2+\va^3\| \mathbf{F}\|_{H_2}^2
+(\curl{ \mathbf{F}},q_x)\},
\end{equation}
where the constant $C$ does not depend on $\va,\ L$.
\end{theorem}

\subsubsection{Preparation}

Before proving Theorem \ref{th.linear1}, we first introduce the following notations:
\begin{eqnarray}
A_1^2&=&\|\sqrt{u_s} v_{y}\|^2+\|\sqrt{u_s} v_{x}\|^2,\label{A1}\\
A_2^2&=&\va (\|\sqrt{u_s}q_{xx}\|^2+\|\sqrt{u_s}q_{xy}\|^2+\|\sqrt{u_s}q_{yy}\|^2)+| u_sq_y(0,\cdot)|^2+ | u_sq_x(L,\cdot)|^2,\nonumber\\\label{A2}
A_3^2&=&\va^3 (\|v_{xxx}\|^2+\|v_{xxy}\|^2+\|v_{xyy}\|^2).\label{A3}
\end{eqnarray}

\begin{lemma}\label{lem.linear.1}
Let $\mau=(u,v)$ be the solution to  system (\ref{linear1})-(\ref{boundary2}), then we have the following estimates:
\begin{eqnarray}
&&\|v\|_{L^2}\lesssim L^{\f14}(\|\sqrt{u_s} v_{y}\|+\|\sqrt{u_s} v_{x}\|),\label{linear.2.1}\\
&&\va^{\f38}\|q_x\|\lesssim A_1 + A_2,   \label{linear3.2.2}\\
&&\va^{\f14}\|v\|_{H^1}+\va\|u_{yy}\|_{L^2}+\va^{\f12+\tau}\|u\|_{L^\infty}+\va^{\f14+\tau}\|v\|_{L^\infty}
+\va^{\f12+\tau}\|v_y\|_{L^\infty}\lesssim \|\mau\|_{\mathcal{X}},  \nonumber\\\label{linear.2.3}
\end{eqnarray}
where $\tau$ is any constant satisfying $0<\tau\ll\gamma$.
\end{lemma}
\begin{remark}\label{re.1}
In fact, (\ref{linear.2.1}) holds for any $v\in H^1(\Omega)$ satisfying $v|_{x=0}=0$, or $v|_{x=L}=0$. For example, we also have
$$\|\nabla q\|\lesssim L^{\f14}\|\sqrt{u_s}\nabla^2q\|.$$
\end{remark}
\begin{proof}
First of all, if $\alpha_1>0$, then $u_s>0$ on $y=2$. Let $\chi$ be the cutoff function defined in (\ref{cutoff}), $\delta>0$ is a small constant to be determined,
 direct computation shows that
\begin{eqnarray}
&&\int_\Omega v^2dxdy=\int_\Omega \chi^2(\f y\delta)v^2dxdy+\int_\Omega [1-\chi^2(\f y\delta)]v^2dxdy\nonumber\\
=&&\int_\Omega \partial_y(y) \chi^2(\f y\delta)v^2dxdy+\int_\Omega [1-\chi^2(\f y\delta)](\int_L^xv_xds)^2dxdy\nonumber\\
=&&-\int_\Omega y[2\chi^2 vv_y +\f2\delta\chi\chi'v^2]dxdy+\int_\Omega [1-\chi^2(\f y\delta)](\int_L^xv_xds)^2dxdy\nonumber\\
\leq&&8\delta\|y^{\f12}\chi v_y \|^2+\f12\|\chi v \|^2+\f {L}\delta\|\sqrt{u_s}v_x\|^2.\label{linear3.2.1}
\end{eqnarray}
If we take $\delta=L^{\f12}$ in (\ref{linear3.2.1}) we obtain
$$\|v\|\lesssim L^{\f14} (\|\sqrt{u_s} v_{y}\|+\|\sqrt{u_s} v_{x}\|).$$
Secondly
\begin{eqnarray}
&&\va^{\f34}\|q_x\|^2\nonumber\\
=&&\va^{\f34}\int_\Omega q_x^2dxdy=\va^{\f34}\int_\Omega \chi^2(\f y{\va^{\f14}})q_x^2dxdy+\va^{\f34}\int_\Omega [1-\chi^2(\f y{\va^{\f14}})]q_x^2dxdy\nonumber\\
=&&\va^{\f34}\int_\Omega \partial_y(y) \chi^2(\f y{\va^{\f14}})q_x^2dxdy+\va^{\f34}\int_\Omega [1-\chi^2(\f y{\va^{\f14}})](\f {v_x}{u_s}-\f{vu_{sx}}{u_s^2})^2dxdy\nonumber\\
=&&-\va^{\f34}\int_\Omega [2y\chi^2 q_xq_{xy} +2y\va^{-\f14}\chi\chi'q_x^2]+\va^{\f34}\int_\Omega [1-\chi^2(\f y{\va^{\f14}})](\f {v_x}{u_s}-\f{vu_{sx}}{u_s^2})^2\nonumber\\
\lesssim&&\va^{\f12+\f3{8}}\|\sqrt{u_s}q_{xy}\|^2\|\chi q_x\|+\|\sqrt{u_s}v_x\|^2
\end{eqnarray}
which implies that
$$\va^{\f38}\|q_x\|\lesssim A_1+A_2.$$
Next, if $\alpha_1=0$, then we have $u_s=0$ on $y=2$. The estimate of $v$ and $q_x$ around $y=2$ can be obtained
by arguments similar to the case near $y=0$. For $u_s=\mu$ satisfying (\ref{0.3}), we can obtain the estimate in a similar manner.
Thus, we obtain (\ref{linear.2.1}) and (\ref{linear3.2.2}).

Finally, recalling the definition of $\mathcal{X}$, the estimate (\ref{linear.2.3}) follows immediately from the divergence free condition
of $\mau$ and the Sobolev imbedding theorem.
\end{proof}

\begin{remark}\label{re.2}
Let's remind here that in the following subsection, we often use the fact that if $u_s=\mu(y),\ v_s=0$, then
$$u_{sx}=0,\ |u_s|_{C^k}\leq C;$$
while if $(u_s,v_s)$ is constructed in section 2, we have
$$|u_{sy}|_{L^\infty}+|u_{syy}|_{L^\infty}\lesssim \alpha_0,$$
and $u_{sx},v_s,v_{sy}$ degenerate near $y=0$ and $y=2$ as shown in (\ref{2.2.3}).
\end{remark}

\subsubsection{Proof of Theorem \ref{th.linear1} }

Let $A_i$, $i=1,2,3$,  be defined in (\ref{A1})-(\ref{A3}), the proof of Theorem \ref{th.linear1} will be broken up into  several lemmas as well as the method of bootstrap.   First we  multiply (\ref{curl1})  with $v$ to have the following weighted estimates for the first order derivatives:

\begin{lemma}\label{lem.2}
 Under the assumption of Theorem \ref{th.linear1}, the solution $(u,v)$ to system (\ref{linear1})-(\ref{boundary2}) satisfying:
\begin{eqnarray}
A_1^2\leq C(\va^{\f{3}8+\f\gamma2}\|\mau\|_{\mathcal{X}}\|\bar{\mathbf{u}}\|_{\mathcal{X}}^2+L^{\f18}(A_2^2+A_3^2)+\| \mathbf{F}\|_{H^1}^2),\nonumber
\end{eqnarray}
where $C$ is a constant independent of $\va$.
\end{lemma}
\begin{proof}
We multiply (\ref{curl1}) with $v$ to see that
\begin{eqnarray}
&&\int_\Omega[ -u_s v_{yy}+u_{syy}v-u_sv_{xx}+u_{sxx}v+S(u,v)-\va\Delta(u_y-v_x)]vdxdy\nonumber\\
=&&\int_\Omega(\curl \mathbf{N}+\curl  \mathbf{F})vdxdy,  \label{linear4}
\end{eqnarray}
whence by virtue of (\ref{linear.2.1}) and the boundary condition (\ref{boundary2}),
\begin{eqnarray}
&&\int_\Omega[ -u_s v_{yy}+u_{syy}v-u_sv_{xx}+u_{sxx}v]vdxdy\nonumber\\
=&&\int_\Omega u_s(v_y^2+v_x^2)dxdy+\int_\Omega u_{sy}vv_ydxdy+\int_\Omega (u_{sxx}+u_{syy})v^2dxdy+\int_\Omega u_{sx}vv_xdxdy\nonumber\\
=&&\int_\Omega u_s(v_y^2+v_x^2)dxdy+\int_\Omega(\f12 u_{syy}v^2+u_{sxx}v^2+ u_{sx}vv_x)dxdy\nonumber\\
\gtrsim&&\|\sqrt{u_s} v_{y}\|^2+\|\sqrt{u_s} v_{x}\|^2.
\end{eqnarray}
Using the divergence free condition, one gets
\begin{eqnarray}
&&-\va\int_\Omega\Delta(u_y-v_x)vdxdy\nonumber\\
=&&\va\int_\Omega u_{yy}v_ydxdy-\va\int_0^2 v_{xx}v|_{x=0}dy-\va\int_\Omega v_{xx}v_xdxdy-2\va\int_\Omega v_{xy}v_ydxdy\nonumber\\
=&&\va\int_\Omega u_yu_{xy}dxdy-\va\int_0^2 v_{xx}v|_{x=0}dy-\f12\va\int_0^2 v_x^2|_{x=L}dy+\va\int_0^2 v_y^2|_{x=0}dy\nonumber\\
=&&\f12\va\int_0^2 u_y^2|_{x=L}+\va\int_0^2 v_y^2|_{x=0}
-\va\int_0^2 v_{xx}v|_{x=0}-\f12\va\int_0^2 v_x^2|_{x=L}\nonumber\\
\gtrsim&&\va|u_y(L,\cdot)|^2+\va|v_y(0,\cdot)|^2-\va|u_s^{-\f14}v_{xx}(0,\cdot)||u_s^{\f14}v(0,\cdot)|-\va|u_sq_x(L,\cdot)|^2, \nonumber
\end{eqnarray}
where
\begin{eqnarray}
|u_s^{\f14}v(0,\cdot)|^2=&&\int_0^2 u_s^{\f12}v^2(0,y)dy=\int_\Omega(2u_s^{\f12}vv_x+\f12u_s^{-\f12}u_{sx}v^2)dxdy\nonumber\\
\leq &&\|\sqrt{u_s}v_{x}\|\|v\|+\va^{\f14}\|v\|^2\leq L^{\f14}A_1^2,\nonumber
\end{eqnarray}
 and
\begin{eqnarray}
&&\va^2|u_s^{-\f14}v_{xx}(0,\cdot)|^2=\va^2\int_0^2 u_s^{-\f12}v_{xx}^2dy=\va^2\int_\Omega\partial_x(u_s^{-\f12}v_{xx}^2)dxdy\nonumber\\
=&&\va^2\int_\Omega[2u_s^{-\f12}v_{xx}v_{xxx}-\f12u_s^{-\f32}u_{sx}v_{xx}^2]dxdy\nonumber\\
\lesssim&&\va^2[(\|\sqrt{u_s}q_{xx}\|+\|q_{x}\|)\|v_{xxx}\|+\va^{\f12}(\|\sqrt{u_s}q_{xx}\|+\|q_{x}\|)^2]\nonumber\\
\lesssim&&\va\|\sqrt{u_s}\nabla^2q\|^2+\va^3\|v_{xxx}\|^2.  \label{linear.2.6}
\end{eqnarray}
Consequently, we conclude
\begin{eqnarray}
-\va\int_\Omega\Delta(u_y-v_x)vdxdy\gtrsim\va|u_y(L,\cdot)|^2+\va|v_y(0,\cdot)|^2-L^{\f18}(A_1^2+A_2^2+A_3^2). \label{linear.2.4}
\end{eqnarray}
On the other hand, it is easy to see that
\begin{eqnarray}
&&|\int_\Omega S(u,v)vdxdy|=|\int_\Omega [-\partial_y(u_{sx}u+v_su_y)+\partial_x(v_sv_y+v_{sx}u)]vdxdy|\nonumber\\
=&&|\int_\Omega (u_{sx}u+v_su_y)v_ydxdy+\int_\Omega (v_sv_{xy}+v_{sxx}u)vdxdy|\nonumber\\
\lesssim && L\va^{\f12}\|\sqrt{u_s}v_y\|^2+L\va\|\sqrt{u_s}v_y\|\|v_{yy}\|+\|v\|_{L^2}(\va\|v_{xy}\|+L\va\|\sqrt{u_s}v_y\|)\nonumber\\
\lesssim && L^{\f14}(\|\sqrt{u_s} v_{y}\|^2+\|\sqrt{u_s} v_{x}\|^2)+\va^2\|\sqrt{u_s}\nabla^2q\|^2.
\end{eqnarray}
Finally, by (\ref{linear.2.3}) we arrive at
\begin{eqnarray}
&&|\int_\Omega(\curl \mathbf{N}+\curl  \mathbf{F})vdxdy|\nonumber\\
=&&|-\va^{\f{11}8+\gamma}\int_\Omega[( \bar{v} \bar{u}_y+ \bar{u}\bar{u}_x)_y
-(\bar{v}\bar{v}_y+ \bar{u}\bar{v}_x)_x]vdxdy+\int_\Omega\curl \mcf vdxdy|\nonumber\\
=&&\va^{\f{11}8+\gamma}|\int_\Omega ( \bar{v}\bar{u}_yv_y-\bar{u}\bar{v}_yv_y-v\bar{v}\bar{v}_{xy}-v\bar{v}_x\bar{v}_y+\bar{u}\bar{v}_xv_x)dxdy+\int_\Omega\curl \mcf vdxdy|\nonumber\\
\leq&&\va^{\f{11}8+\gamma}[\|\bar v\|_{L^\infty}(\|\bar u_y\|\|v_y\|+\|v\|\|\bar{v}_{xy}\|)+\|\bar{u}\|_{L^\infty}\|\nabla v\|\|\nabla \bar v\|+\|v\|_{L^\infty}\|\nabla \bar v\|^2]\nonumber\\
&&+L^{\f14}\|\curl  \mathbf{F}\|A_1\nonumber\\
\lesssim&&\va^{\f{3}8+\f\gamma2}\|\mau\|_{\mathcal{X}}\|\bar{\mathbf{u}}\|_{\mathcal{X}}^2+\| \mathbf{F}\|_{H^1}^2+L^{\f12}A_1^2.\label{linear.2.5}
\end{eqnarray}
Putting (\ref{linear4})-(\ref{linear.2.5}) together, we conclude
\begin{eqnarray}
A_1^2\lesssim\va^{\f{3}8+\f\gamma2}\|\mau\|_{\mathcal{X}}\|\bar{\mathbf{u}}\|_{\mathcal{X}}^2+L^{\f18}(A_2^2+A_3^2)+\| \mathbf{F}\|_{H^1}^2,
\end{eqnarray}
which completes the proof of Lemma \ref{lem.2}.
\end{proof}
Next, to well control the convection terms, we shall use $q_x$ as a multiplier to obtain the second-order derivative estimates.
More precisely, we have the following lemma.
\begin{lemma}\label{lem3.1}
Under the assumptions of Theorem \ref{th.linear1}, the solution $(u,v)$ of the system (\ref{linear1})-(\ref{boundary2}) satisfies
\begin{eqnarray}
A_2^2\lesssim&&LA_1^2+\va^{\f18} A_3^2+\va^{\f\gamma2} \|\mau\|_{\mathcal{X}}\|\bar{\mathbf{u}}\|_{\mathcal{X}}^2+(\curl{ \mathbf{F}},q_x).\label{lem.3.1}
\end{eqnarray}
\end{lemma}
\begin{proof}
We multiply (\ref{curl.2}) with $q_x$ to have
\begin{eqnarray}
&&\int_\Omega[ -\partial_y(u_s^2\partial_yq)-\partial_x(u_s^2q_x)+S(u,v)-\va\Delta(u_y-v_x)]q_xdxdy\nonumber\\
=&&\int_\Omega(\curl \mathbf{N}+\curl  \mathbf{F})q_xdxdy.\nonumber
\end{eqnarray}
First, the convection terms can be bounded as follows.
\begin{eqnarray}
&&\int_\Omega q_x[-\partial_y(u_s^2\partial_yq)-\partial_x(u_s^2q_x)]dxdy\nonumber\\
=&&\int_\Omega u_s^2q_{xy}q_ydxdy-\int_\Omega u_s^2q_xq_{xx}dxdy-\int_\Omega 2u_su_{sx}q_x^2dxdy\nonumber\\
=&&-\f12\int_\Omega u_s^2q_y^2|_{x=0}dy-\f12\int_\Omega u_s^2q_x^2|_{x=L}dy+\f12\int_\Omega u_s^2q_x^2|_{x=0}dy\nonumber\\
&&-\int_\Omega u_su_{sx}(q_x^2+q_y^2)dxdy\nonumber\\
\lesssim&&-\f12\int_\Omega u_s^2q_y^2|_{x=0}dy-\f12\int u_s^2q_x^2|_{x=L}dy+L^{\f12}A_2^2.\label{linear.2.7}
\end{eqnarray}
Now, we turn to controlling the viscous terms:
\begin{eqnarray}
&&-\va\int_\Omega q_x\Delta(u_y-v_x)dxdy\nonumber\\
=&&-\va\int_\Omega u_{yyy}q_xdxdy +\va\int_\Omega(2v_{xyy}+v_{xxx})q_xdxdy\nonumber\\
=&&\va\int_\Omega u_{yy}q_{xy}dxdy -2\va\int_\Omega v_{xy}q_{xy}dxdy-\va\int_\Omega v_{xx}q_{xx}dxdy-\va\int_0^2v_{xx}q_x|_{x=0}dy\nonumber\\
=&&-\va\int_\Omega u_{xyy}q_{y} dxdy-2\va\int_\Omega (u_s q_{xy}+u_{sy}q_x+u_{sx}q_y)q_{xy}dxdy\nonumber\\
&&-\va\int_\Omega (u_s q_{xx}+2u_{sx}q_x+u_{sxx}q)q_{xx}dxdy-\va\int_0^2\f{u_{sx}}{u_s}v_{xx}q|_{x=0}dy\nonumber\\
=&&\va\int_\Omega v_{yyy}q_{y}dxdy  -\va\int_\Omega (2u_s q_{xy}^2+u_s q_{xx}^2)dxdy+\va\int_\Omega u_{syy}q_x^2dxdy\nonumber\\
&&-\va\int_\Omega (2u_{sx}q_xq_{xx}+u_{sxx}qq_{xx}+2u_{sx}q_yq_{xy})dxdy-\va\int_0^2\f{u_{sx}}{u_s}v_{xx}q|_{x=0}dy,  \label{linear.2.2}
\end{eqnarray}
where from Lemma \ref{lem.linear.1} and (\ref{linear.2.6}) one gets
\begin{eqnarray}
&&|\va\int_\Omega [u_{syy}q_x^2-(2u_{sx}q_xq_{xx}+u_{sxx}qq_{xx}+2u_{sx}q_yq_{xy})]dxdy-\va\int_0^2\f{u_{sx}}{u_s}v_{xx}q|_{x=0}dy|\nonumber\\
\lesssim&& \va \|q_x\|^2+\va^{\f32}\|\sqrt{u_s}q_{xx}\|^2+\va^{\f32}\|\sqrt{u_s}q_{xy}\|^2++\va L^{\f12}|u_{sx}u_s^{-\f34}|_{L^\infty}|u_s^{-\f14}v_{xx}(0,\cdot)|\|q_x\|\nonumber\\
\lesssim&& \va L^{\f12} \|\sqrt{u_s}\nabla^2 q\|^2+\va^{\f12+\f18}L^{\f12}(\va^{\f12}\|\sqrt{u_s}\nabla^2q\|+\va^{\f32}\|v_{xxx}\|) \|\sqrt{u_s}\nabla^2 q\|\nonumber\\
\lesssim&& L^{\f12}\|\sqrt{u_s}\nabla^2 q\|^2+\va^{\f14+3} \|v_{xxx}\|^2,
\end{eqnarray}
while
\begin{eqnarray}
&&\va\int_\Omega v_{yyy}q_{y}dxdy\nonumber\\
=&&-\va\int_0^L (u_sq_{yy}+2u_{sy}q_y)q_{y}|_{y=0}dx+\va\int_0^L (u_sq_{yy}+2u_{sy}q_y)q_{y}|_{y=2}dx\nonumber\\
&&-\va\int_\Omega (u_s q_{yy}+2u_{sy}q_y+u_{syy}q)q_{yy}dxdy\nonumber\\
=&&-\va\int_0^L (u_sq_{yy}+u_{sy}q_y)q_{y}|_{y=0}dx+\va\int_0^L (u_sq_{yy}+u_{sy}q_y)q_{y}|_{y=2}dx\nonumber\\
&&-\va\int_\Omega u_s q_{yy}^2+\va\int_\Omega(u_{syy}q_y^2-u_{syy}qq_{yy})dxdy.\label{3.2.4}
\end{eqnarray}

If $u_s>0$ on $y=2$, then we have $q_y=0$, while when $u_s=0$ on $y=2$, we have  $ u_{sy}\leq 0$ on $y=2$. So, we always have
\begin{eqnarray}
\va\int_0^L (u_sq_{yy}+u_{sy}q_y)q_{y}|_{y=2}dx\leq0,
\end{eqnarray}
which  implies
\begin{eqnarray}
&&\va\int_\Omega v_{yyy}q_{y}dxdy  \nonumber\\
&& \leq -\va\int_0^L u_{sy}q_y^2|_{y=0}dx-\va\int_\Omega u_s q_{yy}^2+\va\int_\Omega(u_{syy}q_y^2-u_{syy}qq_{yy})dxdy \nonumber\\
&& \leq -\va\int_0^L u_{sy}q_y^2|_{y=0}dx-\va\int_\Omega u_s q_{yy}^2+\va\|q_y\|^2+\va\|\sqrt{u_s}q_{yy}\|_{L^2}\|u_s^{-\f12}q\| \nonumber\\
&& \leq -\va\int_0^L u_{sy}q_y^2|_{y=0}dx-\va\int_\Omega u_s q_{yy}^2+\va L^{\f14}\|\sqrt{u_s}\nabla^2q\|.\label{linear.2.8}
\end{eqnarray}

Putting (\ref{linear.2.2})-(\ref{linear.2.8}) together, we have
\begin{equation}
-\va\int_\Omega q_x\Delta(u_y-v_x)dxdy
\lesssim -\va \|\sqrt{u_s}\nabla^2q\|^2-\va|q_y(\cdot,0)|_{L^2}^2+\va^{\f14+3}\|v_{xxx}\|^2.\label{linear.2.11}
\end{equation}
Furthermore, by virtue of (\ref{2.2.3}) and Lemma \ref{lem.linear.1}
\begin{eqnarray}
&&|\int_\Omega S(u,v)q_xdxdy|\nonumber\\
=&&|\int_\Omega [-\partial_y(u_{sx}u+v_su_y)+\partial_x(v_sv_y+v_{sx}u)]q_xdxdy|\nonumber\\
=&&|\int_\Omega[(u_{sx}u+v_su_y)q_{xy}+v_{sxx}uq_x-v_{sy}v_xq_x-v_{s}v_xq_{xy}]dxdy|\nonumber\\
\lesssim&&\|\sqrt{u_s}q_{xy}\|(\va^{\f34}\|u\|+\va\|u_y\|+\va\|v_x\|)+\va\|q_{x}\|(\|u\|+\|v_x\|)\nonumber\\
\lesssim&& L(A_2^2+A_1^2).
\end{eqnarray}
Finally, from Lemma \ref{lem.linear.1} one gets
\begin{eqnarray}
&&|\int_\Omega \curl Nq_xdxdy|=\va^{\f{11}8+\gamma}|\int_\Omega q_x[\partial_y(\bar{v}\bar{u}_y-\bar{u}\bar{v}_y)-\partial_x(\bar{v}\bar{v}_y+\bar{u}\bar{v}_x)]dxdy|\nonumber\\
=&&\va^{\f{11}8+\gamma}|-\int_\Omega q_{xy}\bar{v}\bar{u}_ydxdy-\int_\Omega q_x(\bar{u}_y\bar{v}_y+\bar{u}\bar{v}_{yy}+\bar{v}\bar{v}_{xy}+\bar{u}\bar{v}_{xx})dxdy|\nonumber\\
\leq&&\va^{\f{11}8+\gamma}\{\|\sqrt{u_s}q_{xy}\|_{L^2}\|\bar{u}_y\|_{L^2}|\bar{v}_y|_{L^\infty}^{\f12}|\bar{v}|_{L^\infty}^{\f12}+\|q_x\|_{L^2}\|\bar{u}_y\|_{L^2}|\bar{v}_y|_{L^\infty}\nonumber\\
&&+\|q_x\|_{L^2}(|\bar{v}|_{L^\infty}+|\bar{u}|_{L^\infty})\|\nabla^2\bar{v}\|_{L^2}\}\nonumber\\
\lesssim&&\va^{\f\gamma2}\|\mau\|_{\mathcal{X}}\|\bar{\mathbf{u}}\|_{\mathcal{X}}^2.\label{linear.2.10}
\end{eqnarray}
Putting (\ref{linear.2.7})-(\ref{linear.2.10}) together, we obtain the estimate (\ref{lem.3.1}).
\end{proof}

To obtain the estimates for the third order derivatives, we shall first multiply (\ref{curl1}) with $\va^{2} v_{xxx}$
to obtain the estimate for $\|\nabla^2v_x\|$. This will be done in the following lemma.
\begin{lemma}\label{lem3.2}
Under the assumptions of Theorem \ref{th.linear1}, the solution $(u,v)$ to the system (\ref{linear1})-(\ref{boundary2}) satisfies
\begin{eqnarray}
&&A_3^2\leq C\{ A_2^2+\va^{\f38+\gamma}\|\mau\|_{\mathcal{X}}^2\|\bar{\mau}\|_{\mathcal{X}}^2+\va\| \mathbf{F}\|_{H^1}^2\},\nonumber\\\label{linear3.2.4}
\end{eqnarray}
where the constant $C$ does not depend on $\va$.
\end{lemma}

\begin{proof}
We multiply (\ref{curl1})  with $\va^{2} v_{xxx}$   to have\begin{eqnarray}
&& \va^{2}\int_\Omega  v_{xxx}(-u_s v_{yy}+u_{syy}v-u_sv_{xx}+u_{sxx}v+S(u,v))dxdy\nonumber\\&&-\va^3\int_\Omega \Delta (u_y-v_x)v_{xxx}dxdy
=\va^2\int_\Omega(\curl \mathbf{N}+\curl  \mathbf{F})v_{xxx}dxdy.\nonumber
\end{eqnarray}
First of all,
\begin{eqnarray}
&&\va^2|\int_\Omega(-u_s v_{yy}+u_{syy}v-u_sv_{xx}+u_{sxx}v)v_{xxx}dxdy|\nonumber\\
\leq &&C\va^2\|\nabla^2v\|\|v_{xxx}\|\leq\delta A_3^2+C(\delta)A_2^2,\nonumber
\end{eqnarray}
where $\delta>0$ is a sufficiently small constant. By Lemma \ref{lem.linear.1} we see that
\begin{eqnarray}
&&\va^2|\int_\Omega S(u,v)v_{xxx}dxdy|\nonumber\\
=&&\va^2|\int_\Omega[ -\partial_y(u_{sx}u+v_su_y)+\partial_x(v_sv_y+v_{sx}u)]v_{xxx}dxdy|\nonumber\\
=&&\va^2|\int_\Omega[v_{syy}u-v_su_{yy}+v_sv_{xy}+v_{sxx}u]v_{xxx}dxdy|\nonumber\\
\leq&&\va^{2}\left\{\va^{\f12}\|u\|+\va\|u_{y}\|+\va\|u_{xy}\|+\va\|v_{xy}\|\right\}(\|v_{xxx}\|+\|v_{xxy}\|)\nonumber\\
\leq&&\delta A_3^3+\va^{\f12}A_2^2, \nonumber
\end{eqnarray}
where $\delta>0$ is a constant small enough.

Next, we turn to estimating the viscous terms. In view of the boundary condition (\ref{boundary2}), we see that
\begin{eqnarray}
&&\va^{3}\int_\Omega(-u_{yyy}+v_{xxx}+2v_{xyy})v_{xxx}dxdy\nonumber\\
=&&-\va^3\int_\Omega v_{yyyy}v_{xx}dxdy+\va^3\int_\Omega v_{xxx}^2dxdy+2\va^3\int_\Omega v_{xxy}^2dxdy\nonumber\\
=&&\va^3\int_\Omega (v_{xxx}^2+2v_{xxy}^2+v_{xyy}^2)dxdy.\nonumber
\end{eqnarray}
For the nonlinear terms, we have
\begin{eqnarray}
&&\va^2|\int ( \curl N)v_{xxx}|=\va^{2+\f{11}8+\gamma}|\int  v_{xxx}[\bar{v}\bar{u}_{yy}-\bar{u}\bar{v}_{yy}-\bar{v}\bar{v}_{xy}-\bar{u}\bar{v}_{xx})]|\nonumber\\
\leq&&\va^{2+\f{11}8+\gamma}\|v_{xxx}\|[\|\bar{v}\|_{L^\infty}\|\bar{u}_{yy}\|_{L^2}+
(\|\bar{u}\|_{L^\infty}+\|\bar{v}\|_{L^\infty})\|\nabla^2v\|_{L^2}]\nonumber\\
\leq&&\va^{\f{3}8+\gamma}\|\mau\|_{\mathcal{X}}\|\bar{\mau}\|_{\mathcal{X}}^2.\nonumber
\end{eqnarray}

Finally,
\begin{eqnarray}
\va^2|(\curl{\mcf},v_{xxx})|
\leq\va^2 \|\mcf\|_{H^1}\|v_{xxx}\|\leq \delta A_3^2+\va \| \mathbf{F}\|_{H^1}^2.\nonumber
\end{eqnarray}
Thus, (\ref{3.2.4}) follows from combing the above estimates together.
\end{proof}
\begin{remark}
In the case $\alpha_1=0$, we can control $\|\sqrt{u_s}q_{yyy}\|$, and consequently control $\|v_{yyy}\|$  by the method used in \cite{GI}.
For the general case, however,  we can not control $\|\sqrt{u_s}q_{yyy}\|$ as we do not always have good sign for $u_{sy}$
on the upper bound $y=2$.
\end{remark}
Next, we use the forth-order estimates of $v$ as well as Sobolev imbedding to get the bound of $\|v_{yyy}\|$. To this end,
differentiating (\ref{curl1}) with respect to $x$,  we can obtain the following biharmonic equation for $v$
\begin{equation}
\va\Delta^2v=G,\label{linear.2.9}
\end{equation}
where $$G=\partial_x\curl{\mathbf{N}}+\partial_x \mathbf{F}-\partial_x[-u_s\Delta v+\Delta u_{s} v+S(u,v)].$$
Recalling that we already have the following boundary condition for $v$:
$$v_x|_{x=0}=v|_{x=L}=v_{xx}|_{x=L}=0,\ v=v_y=0\ \mbox{ on }\ \{y=0\}\cup\{y=2\},$$
we only need one more condition for $v$ on $x=0$. Using  equation (\ref{curl1}) and boundary condition $u|_{x=0}=v_x|_{x=0}=0$, we can obtain :
 $$\va v_{xxx}(0,y)=[\curl{\mathbf{N}}+\curl \mathbf{F}+u_s\Delta v-\Delta u_{s}v-S(u,v)]|_{x=0}. $$
Besides, it is easy to find that $ v_{xxx}(0,0)= v_{xxx}(0,2)=0$.

 In the following, we shall first homogenize the boundary condition for $v_{xxx}$ on $x=0$.
 Then, making use of the even extension of $v$ with respect to $x=0$, and the odd extension of $v$ with respect to $x=L$,
 we obtain a Dirichlet problem for the biharmonic equation (\ref{linear.2.9}) in a new domain $\Omega^*$.

To homogenize the boundary condition of $v_{xxx}$ on $x=0$, we construct a function $v_0$, such that
$$\partial_xv_0(0,y)=0,\ \partial_{xxx}v_0(0,y)=v_{xxx}(0,y),\  \partial_{y}v_0=0,\ \text{on}\ \{y=0\}\cup\{y=2\}.$$
In fact, we can first define a function $W$ satisfying
\begin{equation}\label{fourth.2}\begin{cases}
\Delta W =0,\ \text{in}\ \Omega,\\
W =0,\ \text{on}\ \{x=L\}\cup\{y=0\}\cup\{y=2\},\\
\partial_xW =v_{xxx}(0,y),\ \text{on}\ \{x=0\}.
\end{cases}\end{equation}
Thanks to the compatibility condition $v_{xxx}(0,0)= v_{xxx}(0,2)=0$, we have $W \in H^2(\Omega)$ satisfying
$$\|W \|_{H^2}\leq C |v_{xxx}(0,\cdot)|_{H^{\f12}([0,2])}$$
Next, we define $v_0$ by solving the following system:
\begin{equation}\label{fourth.1}\begin{cases}
\Delta v_0= W ,\ \text{in}\ \Omega,\\
\partial_yv_0=0,\ \text{on}\ \{y=0\}\cup\{y=2\},\\
\partial_xv_0=0,\ \text{on}\ \{x=0\},\\
v_0=0,\ \text{on}\ \{x=L\}.
\end{cases}\end{equation}
Then, the system (\ref{fourth.1}) has a unique solution $v_0\in H^4(\Omega)$ with the estimate
\begin{equation*}
\|v_0\|_{H^4}\leq C\|W \|_{H^2}\leq C |v_{xxx}(0,\cdot)|_{H^{\f12}([0,2])}.
\end{equation*}
Combining the system (\ref{fourth.2}) with (\ref{fourth.1}), we have
$$\partial_{xxx}v_0(0,y)=\partial_x\Delta v_0(0,y)-\partial_{xyy}v_0(0,y)=v_{xxx}(0,y),$$
$$ v_x(0,y)= \partial_{xxx}v_0(0,0)=\partial_{xxx}v_0(0,2)=0.$$
\begin{remark}
Here we do not need $v_0=0$ on $\{y=0\}\cup\{y=2\}$.
\end{remark}
Let $\chi(x)$ be the cutoff function defined in (\ref{cutoff}), we denote by $\bar v=v-v_0\chi(\f x{L/4})$.
Then $\bar v$ satisfies the following system:
\begin{eqnarray}
\begin{cases}
\Delta^2\bar v=F,\ \text{in}\ \Omega,\\
\bar v_x=\bar v_{xxx}=0,\ \text{on}\ x=0,\\
\bar v=\bar v_{xx}=0,\ \text{on}\ x=L,\\
\bar{v}=-v_0\chi(\f x{L/4}),\ \text{on}\ \{y=0\}\cup\{y=2\},\\
\bar{v}_y=0,\ \text{on}\ \{y=0\}\cup\{y=2\},\\
\end{cases}
\end{eqnarray}
where $F=\va^{-1}G-\Delta^2 (v_0\chi(\f x{L/4}))$.

Now, if we take the even extension of $\bar v$ with respect to $x=0$ and denote the new function by $v_{even}$, then thanks to the boundary condition $\partial_yv_0=0$ on $\{y=0\}\cup\{y=2\}$, as well as the compatibility condition $\partial_{xxx}v_0(0,0)=\partial_{xxx}v_0(0,2)=\partial_xv_0(0,0)=\partial_xv_0(0,2)=0$, we have
$$\partial_yv_{evev}(\cdot,0)=\partial_yv_{evev}(\cdot,2)\equiv0,$$
and
$$|v_{even}(\cdot,0)|_{H^{3+\f12}(-L,L)}+|v_{even}(\cdot,2)|_{H^{3+\f12}(-L,L)}\leq \|v_0\|_{H^4(\Omega)}.$$
Next, following a process similar to that in the proof of Lemma \ref{lem1}, we get
\begin{equation*}
\|\bar v\|_{H^4}\leq\|v_0\|_{H^4}+\|F\|_{L^2}\leq |v_{xxx}(0,\cdot)|_{H^{\f12}([0,2])}+\va^{-1}\|G\|_{L^2},
\end{equation*}
 and  consequently
 \begin{equation}
\va^{\f52}\|v\|_{H^4}\lesssim \va^{\f52}\|\bar v\|_{H^4}+\va^{\f52}\|v_0\|_{H^4}\lesssim \va^{\f52}|v_{xxx}(0,\cdot)|_{H^{\f12}([0,2])}+\va^{\f32}\|G\|_{L^2}.\label{3.2.3}
\end{equation}

Recalling  the expression of $v_{xxx}(0,y)$ and $G$,  we have
\begin{eqnarray}
\va^{\f52}\|v\|_{H^4} \lesssim && \va^{\f32}\|\curl{\mathbf{N}}+\curl \mathbf{F}+u_s\Delta v+\Delta u_{s}v-S(u,v)\|_{H^1}\nonumber\\
\lesssim && \va^{\f32}\|\curl{(\mathbf{N}+ \mathbf{F})}\|_{H^1}+\va^{\f32}\|v\|_{H^3}+\va\|v\|_{L^2}+\va^{\f52}\|u_{yy}\|_{H^1}.\label{fourth.5}
\end{eqnarray}
The Sobolev imbedding theory implies
\begin{equation}
\va^{\f32}\|v\|_{H^3}\leq \va^{\f32}(\delta\va\|v\|_{H^4}+C(\delta)\va^{-1}\|v\|_{H^2})=\delta\va^{\f52}\|v\|_{H^4}+C(\delta)\va^{\f12}\|v\|_{H^2}\label{fourth.3}\end{equation}
here $\delta>0$ is a constant small enough. Besides,
\begin{eqnarray}
\va^{\f32}\|\curl{\mathbf{N}}\|_{H^1}=&&\va^{\f{11}8+\gamma+\f32}\|( \bar{v} \bar{u}_y+ \bar{u}\bar{u}_x)_y
-(\bar{v}\bar{v}_y+ \bar{u}\bar{v}_x)_x\|_{H^1}\nonumber\\
=&&\va^{\f{11}8+\gamma+\f32}\| \bar{v} \bar{u}_{yy}+ \bar{u}\bar{v}_{yy}
-\bar{v}\bar{v}_{xy}+ \bar{u}\bar{v}_{xx}\|_{H^1}\nonumber\\
\lesssim&&\va^{\f{7}8+\f\gamma2}\|\bar{\mau}\|_{\mathcal{X}}^2.\label{fourth.4}
\end{eqnarray}
Putting (\ref{fourth.3}) and (\ref{fourth.4}) into (\ref{fourth.5}), we have
proved the following Lemma:
\begin{lemma}\label{lem3.3}
Under the assumptions of Theorem \ref{th.linear1}, the solution $(u,v)$ to the system (\ref{linear1})-(\ref{boundary2}) satisfies
\begin{equation}
\va^{\f52}\|v\|_{H^4}\leq C(A_2+\va\|\mau\|_{\mathcal{X}}+\va^{\f78+\f\gamma2}\|\bar{\mau}\|_{\mathcal{X}}^2+\va^{\f32}\| \mathbf{F}\|_{H^2}),\label{fourth.7}
\end{equation}
where the constant $C$ does not depend on $\va$.
\end{lemma}

\noindent $\mathbf{Proof\  of \ Theorem}$ \ref{th.linear1}:
First by Lemmas \ref{lem.2}-\ref{lem3.1} and the fact that $L$ is sufficiently small, we find
\begin{eqnarray}
A_1^2+A_2^2\lesssim&& L^{\f14} A_3^2+\va^{\f\gamma2} \|\mau\|_{\mathcal{X}}\|\bar{\mathbf{u}}\|_{\mathcal{X}}^2
+(\curl{ \mathbf{F}},q_x)+\| \mathbf{F}\|_{H^1}^2,
\end{eqnarray}
which, together with Lemma \ref{lem3.2}, implies
\begin{eqnarray}
A_1^2+A_2^2\lesssim \va^{\f\gamma2} \|\mau\|_{\mathcal{X}}\|\bar{\mathbf{u}}\|_{\mathcal{X}}^2+(\curl{ \mathbf{F}},q_x)+\| \mathbf{F}\|_{H^1}^2,\label{3.2.1}
\end{eqnarray}

Next, combing (\ref{fourth.3}) with (\ref{fourth.7}), we see that
\begin{equation}
\va^{\f32}\|v\|_{H^3}\lesssim A_2+\va\|\mau\|_{\mathcal{X}}+\va^{\f78+\f\gamma2}\|\bar{\mau}\|_{\mathcal{X}}^2+\va^{\f32}\| \mathbf{F}\|_{H^2}.\label{fourth.6}
\end{equation}
Finally, using equation (\ref{curl1}), we find
\begin{eqnarray}
\va^{\f32}\|u_{yyy}\|\lesssim &&\va^{\f12}\|v\|_{H^2}+\va^{\f32}\|\nabla^2v_x\|_{L^2}+\va^{\f12}\|\curl{\mathbf{N}}\|_{L^2}+\va^{\f12}\| \mathbf{F}\|_{H^1}\nonumber\\
\lesssim&&A_2+\va\|\mau\|_{\mathcal{X}}+\va^{\f38}\|\bar{\mau}\|_{\mathcal{X}}^2+\va^{\f12}\| \mathbf{F}\|_{H^1},
\end{eqnarray}
which together with (\ref{fourth.6}) gives
\begin{eqnarray}\va^{\f32}\|\mathbf{u}\|_{H^3}\leq A_2+\va\|\mau\|_{\mathcal{X}}+\va^{\f38}\|\bar{\mau}\|_{\mathcal{X}}^2+\va^{\f12}\| \mathbf{F}\|_{H^1}+\va^{\f32}\| \mathbf{F}\|_{H^2}. \label{fourth.8}
\end{eqnarray}

Now, if we put (\ref{3.2.1}) and (\ref{fourth.8}) together, we conclude
\begin{eqnarray*}
\|\mau\|_{\mathcal{X}}^2\lesssim\va^{\gamma}\|\bar{\mathbf{u}}\|_{\mathcal{X}}^4+(\curl{ \mathbf{F}},q_x)+\| \mathbf{F}\|_{H^1}^2+\va^{3}\| \mathbf{F}\|_{H^2}^2.
\end{eqnarray*}
This completes the proof of Theorem \ref{th.linear1}.

\renewcommand{\theequation}{\thesection.\arabic{equation}}
\setcounter{equation}{0}
\section{Proof of the main theorems}

 Based on the uniform-in-$\va$ estimates for the linearized system as well as the contract mapping principle,
 we are now ready to prove our main theorems. Obviously, the main difference between the proofs of the main theorems
 lies in the term $(\curl{ \mathbf{F}},q_x)$.
 \bigskip

\noindent $\mathbf{Proof\ of\ the\  main\  theorems:}$

Step I: Estimates of the term $(\curl{ \mathbf{F}},q_x)$

(i) In the Poiseuille-Couette flow case with $\alpha_2>0$ in the absence of external forces, we take
$$ \mau_s=\mau^0, \quad P_s=\va U''(y) x=-2\va\alpha_2x, $$
where $U$ is defined in (\ref{0.2}).
Then, by (\ref{remain1}) and (\ref{remain2}),
\begin{eqnarray*}
&&\mcf_u=-\va^{-\f{11}8-\gamma}(-\va\mu''+\partial_xP_s)=\va^{-\f{11}8-\gamma}\va(\mu''-U''),\\
&&\mcf_v=0.
\end{eqnarray*}
 Thus, from the condition (\ref{0.13}) in Theorem \ref{main.1} we get
 \begin{eqnarray}
 &&|(\curl{ \mathbf{F}},q_x)|=|(\curl{\mcf},q_x)|\nonumber\\
 =&&|\va^{-\f38-\gamma}\int_\Omega\mu'''(y)q_xdxdy|= |\va^{-\f38-\gamma}\int_0^2\mu'''q(0,y)dy|\nonumber\\
 \leq &&C\va^{-\f38-\gamma}|\mu'''|_{L^\infty}(|u_sq_y(0,\cdot)|_{L^2}+\|\sqrt{u_s}v_x\|_{L^2}), \nonumber\\
 \leq&&C\alpha_0\|\mau\|_{\mathcal{X}},\label{4.1}
 \end{eqnarray}
where the constant $C$ does not depend on $\va$.

(ii) In the case when $\alpha_2=0$ in the absence of external forces, i.e., the flow is of Couette type,
$(\mau_s,P_s)$ is defined in (\ref{asy:1:1:1}) and $(\mathcal{F}_u,\mathcal{F}_v)$ can be written as (\ref{2.2.6}) and (\ref{2.2.7}).
Therefore, by virtue of (\ref{2.2.8}) and (\ref{linear3.2.2}), one has
\begin{eqnarray}
 &&|(\curl{ \mathbf{F}},q_x)|=|(\curl{\mcf},q_x)|\nonumber\\
 \leq &&\|\curl{\mcf}\|\|q_x\|\leq C\va^{-\f{11}8-\gamma}\va^{2-\f38}\alpha_0\|\mau\|_{\mathcal{X}}=C\va^{\f14-\gamma}\alpha_0\|\mau\|_{\mathcal{X}}, \label{4.2}
\end{eqnarray}
with the constant $C$ being independent of $\va$.

(iii) In the case when there is a proper external force $\mathbf{f}^\va=(f^\va_1,f^\va_2)$ to control the flow, we take
$$\mau_s=\mau^0,\quad P_s=C,$$
where $C$ is any constant. By (\ref{remain1}) and (\ref{remain2}) we have
$$F_1=\va^{-M_0}f^\va_1+\mcf_u=\va^{-M_0} (\va\mu''-f^\va_1)= \va^{-M_0}g_1^\va;$$
$$ F_2=\va^{-M_0}f^\va_2+\mcf_v=\va^{-M_0} f^\va_2=\va^{-M_0} g_2^\va.$$
Then, from (\ref{0.12}) in Theorem \ref{main.3} we get
\begin{eqnarray}
|(\curl{\mathbf{F}},q_x)|=&&\va^{-\f{11}8-\gamma}|(\curl{\mathbf{g}^\va},q_x)|=\va^{-\f{11}8-\gamma}|\int_0^2\curl{\mathbf{g}^\va}q|_{x=0}dy| \nonumber\\
\leq &&\va^{-\f{11}8-\gamma}\|\mathbf{g}^\va\|_{H^2}(|u_sq_y(0,\cdot)|_{L^2}+\|\sqrt{u_s}v_x\|_{L^2})\nonumber\\
\leq &&C\alpha_0\|\mau\|_{\mathcal{X}},\label{4.3}
\end{eqnarray}
where the constant $C$ does not depend on $\va$.
\vspace{1mm}

Step II: Solutions to the nonlinear systems.

First of all, let $(\mau_s,P_s)$ be one of the cases in Step I,
 then we see that
\begin{equation}
\|\mathcal{\mathbf{F}}\|_{H_1}^2+\va^3\|\mathcal{\mathbf{F}}\|_{H_2}^2+|(\curl{\mathbf{F}},q_x)|\leq C(\alpha_0^2+\alpha_0\|\mau\|_{\mathcal{X}})\label{4.4}\end{equation}
with the constant $C$ being independent of $\va$. So we have from  Theorem \ref{th.linear1}
\begin{equation}
\|\mau\|_{\mathcal{X}}^2\leq C(\va^{\gamma }\|\bar{\mau}\|_{\mathcal{X}}^4+\alpha_0^2).\label{4.0}\end{equation}
In view of Theorems \ref{th.linear2} and (\ref{4.0}), we see that
for any $\bar{\mau}=(\bar u,\bar v)\in \mathcal{X}$, there is a unique solution $(\mau,P)$ satisfying (\ref{linear1})-(\ref{boundary2})
with the following estimate:
\begin{equation}
\|\mau\|_{\mathcal{X}}\leq C(\va^{\f\gamma 2}\|\bar{\mau}\|_{\mathcal{X}}^2+\alpha_0).\label{4.5}\end{equation}

Now, for $\alpha_0$ small enough, we define
$$V=\{\mau\in\mathcal{X}|\|\mau\|_{\mathcal{X}}<2C\alpha_0\}.$$
Then we can define a mapping $T:V\rightarrow V$ by $T(\bar\mau)=\mau$.

Now, for any $\bar\mau_1,\bar\mau_2\in V$, we denote $\mau_1=(u_1,v_1)=T(\bar\mau_1)$ and $\mau_2=(u_2,v_2)=T(\bar\mau_2)$.
Then, $\mau_1-\mau_2$ satisfies the following system:
\begin{eqnarray}
&&-u_s \Delta (v_1-v_2)+\Delta u_s(v_1-v_2)+S(u_1-u_2,v_1-v_2)-\va\Delta\curl (\mau_1-\mau_2)\nonumber\\
=&&\curl \mathbf{N}(\bar u_1,\bar v_1)-\curl \mathbf{N}(\bar u_2,\bar v_2)\nonumber\\
=&&\partial_y[(\bar{v}_1-\bar{v}_2)\bar{u}_{1y}+(\bar{u}_1-\bar{u}_2)\bar{u}_{1x}]-\partial_x[(\bar{v}_1-\bar{v}_2)\bar{v}_{1y}+(\bar{u}_1-\bar{u}_2)\bar{v}_{1x}]\nonumber\\
&&+\partial_y[\bar v_2(\bar u_1-\bar u_2)_y+\bar u_2(\bar u_1-\bar u_2)_x]-\partial_x[v_2(\bar v_1-\bar v_2)_y+\bar u_2(\bar v_1-\bar v_2)_x].\label{4.6}
\end{eqnarray}
If we further denote $\bar q=\f{v_1-v_2}{u_s}$, and
\begin{eqnarray}
B_1^2&=&\|\sqrt{u_s} (v_1-v_2)_{y}\|^2+\|\sqrt{u_s} (v_1-v_2)_{x}\|^2,\nonumber\\
B_2^2&=&\va \|\sqrt{u_s}\bar{q}_{xx}\|^2+\va\|\sqrt{u_s}\bar{q}_{xy}\|^2+\va\|\sqrt{u_s}\bar{q}_{yy}\|^2+|u_s\bar q_y(0,\cdot)|^2 , \nonumber
\end{eqnarray}
then we have in a manner similar to (\ref{3.2.1}) that
\begin{equation}
B_1^2+B_2^2\leq \va^{\f\gamma2}\|\mau_1-\mau_2\|_{\mathcal{X}}\|
\bar \mau_1-\bar \mau_2\|_{\mathcal{X}}(\|\bar{\mau}_1\|_{\mathcal{X}}+\|\bar\mau_2\|_{\mathcal{X}}).\label{4.7}\end{equation}
Similarly to the proof of Lemma \ref{lem3.3}, we obtain
\begin{equation}
\va^{\f52}\|v_1-v_2\|_{H^4}\lesssim B_2+\va\|\mau_1-\mau_2\|_{\mathcal{X}}+\va^{\f78}\|\bar\mau_1-\bar\mau_2\|_{\mathcal{X}}(\|\bar{\mau}_1\|_{\mathcal{X}}
+\|\bar\mau_2\|_{\mathcal{X}}).\label{4.8}
\end{equation}
Combing (\ref{4.7}) with (\ref{4.8}), we argue similarly to the proof of Theorem \ref{th.linear1} to deduce
$$\|\mau_1-\mau_2\|_{\mathcal{X}}\lesssim \va^{\f\gamma2}\|\bar \mau_1-\bar \mau_2\|_{\mathcal{X}}(\|\bar{\mau}_1\|_{\mathcal{X}}+\|\bar\mau_2\|_{\mathcal{X}}).$$
Thus, $T$ is a compact mapping if $\alpha_0$ is taken to be small enough. We finish the proof of the main theorems.

\renewcommand{\theequation}{\thesection.\arabic{equation}}
\setcounter{equation}{0}
\appendix
\section{Sketch of the construction of $\mau_p^{i,-}$}

In this section, we sketch the construction of $\mau_p^{i,-}$, and refer to \cite{IZ1} for the details.

As the equations for the weak boundary layer correctors are linear parabolic equations including  degenerate terms $\mu(y)\partial_xu_p^{i,-}$
near $y=0$, we would have to use the scale $Y_-=\va^{-\f13}y$ near $y=0$.
 Recalling that $\mu(y)\sim y$ as $y\rightarrow0$, and that the boundary layer correctors would degenerate rapidly when $Y_->1$, we first collect from (\ref{remain1}) the leading   $\bigO(\eps^{\frac{4}{3}})$ terms for the boundary layer profiles near $y=0$:
\begin{eqnarray}
R^{1,-}_u=\eps^{-\f13}\mu\partial_xu_p^{1,-}+v_p^{1,-}\mu'+\eps^{-\f13}\partial_xP_p^{1,-}-\partial_{Y_-Y_-}u_p^{1,-},\end{eqnarray}
and the leading $\bigO(\eps^{\frac{2}{3}})$ terms for (\ref{remain2}) is
\begin{equation}
\partial_{Y_-}P_p^{1,-}=0.
\end{equation}

Then to construct $\mau_p^{1,-}$, we first consider the following system (using the fact that $\mu(y)\sim y$ when $y\rightarrow0$):
\begin{eqnarray}\begin{cases}
Y_-\partial_xu_p^{1,0,-}+v_p^{1,0,-}+\eps^{-\f13}\partial_xP_p^{1,0,-}-\partial_{Y_-Y_-}u_p^{1,0,-}=0,\ \ \partial_{Y_-}P_p^{1,0,-}=0,\\
u_p^{1,0,-}|_{x=0}=0,\ u_p^{1,0,-}|_{Y_-=0}=-u_e^1(x,0),\ u_p^{1,0,-}|_{Y_-\rightarrow\infty}=0,\\
v_p^{1,0,-}=\int_{Y_-}^\infty \partial_xu_p^{1,0,-}.
\end{cases}\end{eqnarray}
It is easy to obtain $P_p^{1,-}\equiv0$, and thus,
\begin{eqnarray}\begin{cases} \label{eq:PR:1}
Y_-\partial_xu_p^{1,0,-}+v_p^{1,0,-}-\partial_{Y_-Y_-}u_p^{1,0,-}=0,\\
u_p^{1,0,-}|_{x=0}=0,\ u_p^{1,0,-}|_{Y_-=0}=-u_e^1(x,0),\ u_p^{1,0,-}|_{Y_-\rightarrow\infty}=0,\\
v_p^{1,0,-}=\int_{Y_-}^\infty \partial_xu_p^{1,0,-}.
\end{cases}\end{eqnarray}

 Note that we construct $u^{1,0,-}_p, v^{1,0,-}_p$ on $(0,L) \times (0,\infty)$. We now cut-off these layers and
 make an $O(\va^{\f13})$-order error:
\begin{align} \label{cut.off.1}
u^{1,-}_p = \chi(\f{\eps^{\f13}Y_-}{a_0}) u^{1,0, -}_p - \frac{\eps^{\f13}}{a_0} \chi'(\f{\eps^{\f13}Y_-}{a_0}) \int_0^x v^{1,0,-}_p,
\;\; v^{1,-}_p := \chi(\f{\eps^{\f13}Y_-}{a_0}) v^{1,0, -}_p
\end{align}
where $a_0>0$ is a fixed constant small enough, and $\chi$ is defined in (\ref{cutoff}).

When $i=2$, let $u_p^{1,0,-},\ v_p^{1,0,-}$ be solutions to the system (\ref{eq:PR:1}), then after cutting off (\ref{cut.off.1}),
the contribution to the next layer is
\begin{eqnarray}
\mathcal{C}_{cut}^{1,-}=\f1{a_0}\eps^{\f13}Y_-\chi'v_p^{1,0,-}+3\f1{a_0}\eps^{\f13}\chi'\partial_{Y_-}u_p^{1,0,-}
+3\f1{a_0^2}\eps^{\f23}\chi''u_p^{1,0,-}+\f1{a_0^3}\eps\chi'''\int_0^xv_p^{1,0,-}.\nonumber\\
\end{eqnarray}
We also obtain another error, due to approximating $\eps^{- \frac 1 3} \mu$ by $Y_-$ in the support of the cut-off function $\chi(\frac{\eps^{\frac 1 3} Y}{a_0})$ and by approximating $\mu'$ by $1$. This error is given by
\begin{align}
\mathcal{C}_{approx}^{1,-} := &  (\eps^{- \frac 1 3} \mu(y) - Y_-)(\chi(\frac{y}{a_0}) \p_x u^{1,0,-}_p+\f1{a_0}\eps^{\f13}\chi'u_p^{1,0,-}) + (\mu' - 1) \chi(\frac{y}{a_0}) v^{1,0,-}_p
\end{align}
Finally, we have higher order terms that contribute to the error:
\begin{eqnarray}
\mathcal{C}^{1,-}_{quad} := &&\eps^2 (u^1_e + u^{1,-}_p) \p_x u^{1,-}_p + \eps^2 u^{1,-}_p \p_x u^1_e + \eps^{\frac 53} v^1_e u^{1,-}_{pY_-} + \eps^{\frac 73} v^{1,-}_p u^1_{ey}\nonumber\\
&& + \eps^2 v^1_p u^{1,-}_{pY_-} - \eps^2 u^{1,-}_{pxx}.
\end{eqnarray}
For the higher order terms in the second equation, we shall use our auxiliary pressure to move it to the top equation.
This is achieved by defining the first auxiliary pressure $P^{1,a,-}_p$ to  zero out the terms contributed from
\begin{align} \n
&\eps^{\frac 43} (\mu + \eps u^1_e) v^{1,-}_{px} + \eps u^{1,-}_p (\eps v^1_{ex} + \eps^{\frac 4 3} v^{1,-}_{px})
+ \eps^2 v^1_e v^{1,-}_{pY_-} + \eps^{\frac 73} v^1_p v^{1,-}_{pY_-} + \eps^{\frac 73} v^{1,-}_p v^1_{ey} \\
&- \eps^{\frac 53} v^{1,-}_{pY_-Y_-} - \eps^{\frac 73} v^{1,-}_{pxx} + \eps^{\frac 4 3} P^{1,a,-}_{pY_-} =0,
\end{align}
which therefore motivates our following definition:
\begin{align}  \n
- \eps^{\frac 4 3} P^{1,a,-}_{P} := & \int_{Y_-}^{\infty} \Big( \eps^{\frac 43} (\mu + \eps u^1_e) v^{1,-}_{px} + \eps u^{1,-}_p (\eps v^1_{ex} + \eps^{\frac 4 3} v^{1,-}_{px}) + \eps v^1_e v^{1,-}_{pY_-} + \eps^{\frac 73} v^{1,-}_p v^{1,-}_{pY_-}  \\
&+ \eps^{\frac 73} v^{1,-}_p v^1_{ey} - \eps^{\frac 53} v^{1,-}_{pY_-Y_-} - \eps^{\frac 73} v^{1,-}_{pxx} \Big) \ud Y'.
\end{align}
As a result, we can define the force for the next order weak boundary layer via
\begin{align}
F^{2,-} := \eps^{- \frac 5 3} \Big( -\eps^{\frac 4 3} \mathcal{C}^{1,-}_{cut} +  \eps^{\frac 4 3} \mathcal{C}^{1,-}_{approx} + \mathcal{C}^{1,-}_{quad} - \eps^{\frac 5 3} \p_x P^{1,a,-}_P \Big).
\end{align}

\begin{remark} For $i=3,...,M$, $F^{i,-}$ can be constructed similarly.
Since we shall put the terms with $H^2$-norm smaller than $\va^{\f74+\gamma}$ into the remainders $(\mcf_u,\mcf_v)$,
we only need the auxiliary pressure $P_p^{i,a,-}$ for $i=1,...,5$. When $i>5$, we shall take $P_p^{i,a,+}=0$.
We remark here that all the interaction terms with form $u_e^{i,+}u_p^{j,-}$ will be put into the remainders $(\mcf_u,\mcf_v)$.
\end{remark}

Let us therefore consider the abstract problem (dropping indices):

\begin{eqnarray}
&&Y \p_x u + v - \p_{Y}^2 u = F, \qquad (x, Y) \in (0, L) \times (0, \infty) \label{A.1}\\
&&v := \int_Y^\infty \p_x u \ud Y',\label{A.3} \\
&&u|_{x = 0} = 0, \qquad u|_{Y = 0} = g(x), \qquad u|_{Y \rightarrow \infty} = 0.\label{A.4}
\end{eqnarray}
For this abstract problem, we have the following estimates:
\begin{lemma} \label{A.lemma1} Assume that $F(x,Y)$ decays rapidly at infinity, i.e., for any $m\geq0$, there is a constant $M>0$,
such that
\begin{equation*}
\|(1+Y)^m \p_x^n \p_Y^{l}F \| \leq M  \ \text{ for } 0 \le 2n+l \le K,
\end{equation*}
where $K>0$ is a sufficiently large constant. Then, there exists a unique solution $(u, v)$ to \eqref{A.1}--\eqref{A.4} satisfying
\begin{equation*}
\|(1+Y)^m \partial_x^n \partial_Y^{l} \{ u, v \}\| \leq C(m,n,l)(M+\|g\|_{H^{K+1}})\ \text{ for any } 2n + l \leq K+2,
\end{equation*}
where the constant $C$ does not depend on $Y$.
\end{lemma}

For $i = M$, we need to slightly modify the abstract problem. Consider
\begin{align} \nonumber
&Y \p_x u + v - \p_{Y}^2 u = F, \qquad (x, Y) \in (0, L) \times (0, \infty) \\ \label{final.2}
&v := - \int_0^Y \p_x u \ud Y', \\ \nonumber
&u|_{x = 0} = 0, \qquad u|_{Y = 0} = g(x), \qquad \partial_Yu|_{Y \rightarrow \infty} = 0,
\end{align}
and we have
\begin{lemma}\label{A.lemma2}
Assume  that $F(x,Y)$ decays rapidly at infinity, i.e., for any $m\geq0$, there is a constant $M>0$,
such that
\begin{align}
\|(1+Y)^m \p_x^n \p_Y^{l}F \| \leq M \text{ for } 0 \le 2n+ l \le K ,
\end{align}
then there exists a unique solutions $(u, v)$ of (\ref{final.2}), satisfying
\begin{align*}
&\|(1+Y)^m \p_x^n \p_Y^{l} u\| \le  C(m,n,l)(M+\|g\|_{H^{K+1}}) \text{ for any }  l \ge 1, 1 \le 2n + l \le K+2, \\
&\|\p_x^n \{ u, \frac{v}{Y} \}\| \le  C(m,n,l)(M+\|g\|_{H^{K+1}}) \text{ for any } 0 \le 2n \le K+2.
\end{align*}
\end{lemma}
\begin{remark} The main difference between the proofs of Lemma \ref{A.lemma1} and  Lemma \ref{A.lemma2} lies in that one does not
have the decay of $v$ as $Y\rightarrow\infty$ in Lemma \ref{A.lemma2}.
\end{remark}
\bigskip

\noindent\textbf{Acknowledgements.} Chunhui Zhou would like to thank Prof. Yan Guo for many fruitful discussions.
The research
of Song Jiang is supported by National Key R$\&$D Program (2020YFA0712200), National Key Project
(GJXM92579), and the Sino-German Science Center (Grant No.
GZ 1465) and the ISF¨CNSFC joint research program (Grant No. 11761141008).
  Chunhui Zhou is supported by the NSFC under the contract 11871147 and the Zhishan scholarship of Southeast University.

\end{document}